\documentclass[12pt]{amsart}

\usepackage{graphicx}
\usepackage{subfigure}
\usepackage{eepic}
\usepackage{xcolor}
\selectcolormodel{gray}
\usepackage{tikz}
\usepackage{amsmath}
\usepackage{a4wide}
\usepackage[utf8]{inputenc}
\usepackage{amssymb}
\usepackage{amsopn}
\usepackage{epsfig}
\usepackage{amsfonts}
\usepackage{latexsym}
\usepackage{amsthm}
\usepackage{enumerate}
\usepackage[UKenglish]{babel}
\usepackage{verbatim}
\usepackage{color}

\newtheorem{theorem}{Theorem}[section]
\newtheorem{lemma}[theorem]{Lemma}
\newtheorem{proposition}[theorem]{Proposition}
\newtheorem{corollary}[theorem]{Corollary}

\theoremstyle{definition}

\theoremstyle{remark}
\newtheorem{remark}[theorem]{Remark}

\numberwithin{equation}{section}

\renewcommand{\b}{\beta}

\begin{document}

\title{Digit frequencies and self-affine sets with non-empty interior}

\author{Simon Baker}
\address{Mathematics institute, University of Warwick, Coventry, CV4 7AL, UK}
\email{simonbaker412@gmail.com}

\date{\today}


\begin{abstract}
In this paper we study digit frequencies in the setting of expansions in non-integer bases, and self-affine sets with non-empty interior.

Within expansions in non-integer bases we show that if $\beta\in(1,1.787\ldots)$ then every $x\in(0,\frac{1}{\beta-1})$ has a simply normal $\beta$-expansion. We also prove that if $\beta\in(1,\frac{1+\sqrt{5}}{2})$ then every $x\in(0,\frac{1}{\beta-1})$ has a $\beta$-expansion for which the digit frequency does not exist, and a $\beta$-expansion with limiting frequency of zeros $p$, where $p$ is any real number sufficiently close to $1/2$.

For a class of planar self-affine sets we show that if the horizontal contraction lies in a certain parameter space and the vertical contractions are sufficiently close to $1,$ then every nontrivial vertical fibre contains an interval. Our approach lends itself to explicit calculation and give rise to new examples of self-affine sets with non-empty interior. One particular strength of our approach is that it allows for different rates of contraction in the vertical direction.
\end{abstract}

\maketitle

\section{Introduction}\label{sec:1}
Let $x\in[0,1]$. A sequence $(\epsilon_i)\in\{0,1\}^{\mathbb{N}}$ is called a \emph{binary expansion} of $x$ if $$x=\sum_{i=1}^{\infty}\frac{\epsilon_i}{2^{i}}.$$ It is well known that apart from the dyadic rationals (numbers of the form $p/2^n$) every $x\in[0,1]$ has a unique binary expansion. The exceptional dyadic rationals have precisely two binary expansions. A seemingly innocuous generalisation of these representations is to replace the base $2$ with a parameter $\beta\in(1,2)$. That is, given $x\in \mathbb{R},$ we call a sequence $(\epsilon_i)\in\{0,1\}^{\mathbb{N}}$ a \emph{$\b$-expansion} of $x$ if $$x=\pi_{\beta}((\epsilon_i)):=\sum_{i=1}^{\infty}\frac{\epsilon_{i}}{\beta^{i}}.$$ These representations were first introduced in the papers of Parry \cite{Parry} and R\'{e}nyi \cite{Renyi}. It is straightforward to show that $x$ has a $\beta$-expansion if and only if $x\in [0,\frac{1}{\beta-1}]$. In what follows we will let $I_{\beta}:=[0,\frac{1}{\beta-1}]$.

Despite being a simple generalisation of binary expansions, $\b$-expansions exhibit far more exotic behaviour. In particular, one feature of $\beta$-expansions that makes them an interesting object to study, is that an $x\in I_{\beta}$ may have many $\beta$-expansions. In fact a result of Sidorov \cite{Sid2} states that for any $\beta\in(1,2),$ Lebesgue almost every $x\in I_{\beta}$ has a continuum of $\beta$-expansions. Moreover, for any $k\in \mathbb{N}\cup\{\aleph_{0}\},$ there exists $\beta\in(1,2)$ and $x\in I_{\beta}$ such that $x$ has precisely $k$ $\beta$-expansions, see \cite{Bak2,BakerSid,EHJ,EJ,Sid1}. Note that the endpoints of $I_{\beta}$ always have a unique $\beta$-expansion for any $\beta\in(1,2)$.

A particularly useful technique for studying both binary expansions and $\beta$-expansions is to associate a dynamical system to the base. One can then often reinterpret a problem in terms of a property of the dynamical system. The underlying geometry of the dynamical system can then make a problem much more tractable. In this paper we prove results relating to digit frequencies and self-affine sets. These results are of independent interest but also demonstrate the strength of the dynamical approach to $\beta$-expansions.

\section{Statement of results}

\subsection{Digit frequencies}
Let $(\epsilon_{i})\in\{0,1\}^{\mathbb{N}}.$ We define the \emph{frequency of zeros of $(\epsilon_i)$} to be the limit $$\textrm{freq}_{0}(\epsilon_i):=\lim_{n\to\infty}\frac{\#\{1\leq i \leq n: \epsilon_i=0\}}{n}.$$ Assuming the limit exists. We call a sequence $(\epsilon_i)$ \emph{simply normal} if $\textrm{freq}_{0}(\epsilon_{i})=1/2.$ For each $x\in[0,1],$ we let $\textrm{freq}_{0}(x)$ denote the frequency of zeros in its binary expansion whenever the limit exists. When the limit does not exist we say $\textrm{freq}_{0}(x)$ does not exist. The following results are well known:
\begin{enumerate}
  \item Lebesgue almost every $x\in[0,1]$ has a simply normal binary expansion.
  \item $ \dim_{H}(\{x: \textrm{freq}_{0}(x)\textrm{ does not exist}\})=1.$
  \item For each $p\in [0,1]$ we have
  $$\dim_{H}(\{x: \textrm{freq}_{0}(x)=p\})=\frac{-p\log p -(1-p)\log(1-p)}{\log 2}.$$
\end{enumerate}
In $(3)$ we have adopted the convention $0\log 0=0$. The first statement is a consequence of Borel's normal number theorem \cite{Bor}, the second statement appears to be folklore, and the third statement is a result of Eggleston \cite{Egg}.

The above results provide part of the motivation for the present work. In particular, we are interested in whether analogues of these results hold for expansions in non-integer bases. Our first result in the setting of $\beta$-expansions is the following.

\begin{theorem}
\label{frequency theorem}
\begin{enumerate}
  \item Let $\beta\in(1,\beta_{KL})$. Then every $x\in(0,\frac{1}{\beta-1})$ has a simply normal $\beta$-expansion.
  \item Let $\beta\in(1,\frac{1+\sqrt{5}}{2})$. Then every $x\in(0,\frac{1}{\beta-1})$ has a $\beta$-expansion for which the frequency of zeros does not exist.
  \item Let $\beta\in(1,\frac{1+\sqrt{5}}{2})$. Then there exists $c=c(\beta)>0$ such that for every $x\in (0,\frac{1}{\beta-1})$ and $p\in[1/2-c,1/2+c],$ there exists a $\b$-expansion of $x$ with frequency of zeros equal to $p$.
\end{enumerate}
\end{theorem}
The quantity $\beta_{KL}\approx 1.787$ appearing in statement $1$ of Theorem \ref{frequency theorem} is the Komornik-Loreti constant introduced in \cite{KomLor}. In \cite{KomLor} Komornik and Loreti proved that $\beta_{KL}$ is the smallest base for which $1$ has a unique $\beta$-expansion. It has since been shown to be important for many other reasons. We elaborate on the significance of this constant and its relationship with the Thue-Morse sequence in Section \ref{Section3}. Note that we can explicitly calculate a lower bound for the quantity $c$ appearing in statement $3$ of Theorem \ref{frequency theorem}. We include some explicit calculations in Section \ref{Section Explicit calculations}.

It follows from the results listed above that the set of $x$ whose binary expansion is not simply normal has Hausdorff dimension $1$. Our next result shows that as $\beta$ approaches $2$ we see a similar phenomenon.
\begin{theorem}
\label{exceptional theorem}
$$\lim_{\beta\nearrow 2}\dim_{H}\Big(\Big\{x: x \textrm{ has no simply normal } \beta\textrm{-expansion}\Big\}\Big)=1.$$
\end{theorem}

\subsection{Hybrid expansions}
In this section we consider $\beta$-expansions where our digit set is $\{-1,1\}$ instead of $\{0,1\}$. Given $\beta\in(1,2)$ and $x\in [\frac{-1}{\beta-1},\frac{1}{\beta-1}],$ we say that a sequence $(\epsilon_i)\in\{-1,1\}^{\mathbb{N}}$ is a \emph{hybrid expansion of $x$ }if the following holds: $$x=\sum_{i=1}^{\infty}\frac{\epsilon_i}{\b^i}$$ and $$x=\lim_{n\to\infty} \frac{1}{n}\sum_{i=1}^{n}\epsilon_i.$$ Hybrid expansions were first introduced by G\"{u}nt\"{u}rk in \cite{Gun}. Interestingly, the original motivation for studying hybrid expansions was to overcome the problem of analogue to digital conversion where the underling system has background noise. In \cite{Gun} the following result was asserted without proof.

\begin{theorem}
\label{DKK theorem}
There exists $C_1>0,$ such that for all $\beta\in(1,1+C_1)$ there exists $c=c(\beta)>0,$ such that every $x\in [-c,c]$ has a hybrid expansion.
\end{theorem}
A proof was subsequently provided by Dajani, Jiang, and Kempton in \cite{DKK}. They showed that one can take $C_{1}\approx 0.327.$ We improve upon this theorem in the following way.

\begin{theorem}
\label{hybrid theorem}
Let $\beta\in(1,\frac{1+\sqrt{5}}{2})$. Then there exists $c= c(\beta)>0$ such that every $x\in [-c,c]$ has a hybrid expansion.
\end{theorem}
It would be desirable to obtain a result of the form: there exists $C>0$ such that for every $\beta\in(1,1+C)$ every $x\in(-\frac{1}{\beta-1},\frac{1}{\beta-1})$ has a hybrid expansion. However, it is an immediate consequence of the definition that if $x$ has a hybrid expansion then $x\in[-1,1]$. Since $[-1,1]\subsetneq(\frac{-1}{\beta-1},\frac{1}{\beta-1})$ for all $\beta\in(1,2)$ it is clear that such a result is not possible. Note that if we normalised by a function that decayed at a slower rate than $n^{-1}$ we would not necessarily have this obstruction. The following result shows that if we replace $n^{-1}$ with another normalising function that satisfies a certain growth condition, then we have our desired result.

\begin{theorem}
\label{slow growth theorem}
Let $\beta\in(1,\frac{1+\sqrt{5}}{2}).$ Then there exists $c=c(\beta)>0$ such that if $f:\mathbb{N}\to (0,\infty)$ is a strictly increasing function which satisfies $$\limsup_{n\to\infty}f(n+1)-f(n)<c$$and $$\lim_{n\to\infty}f(n)= \infty,$$ then for every $x\in (-\frac{1}{\beta-1},\frac{1}{\beta-1})$ there exists $(\epsilon_i)\in\{-1,1\}^{\mathbb{N}}$ such that $$x=\sum_{i=1}^{\infty}\frac{\epsilon_{i}}{\beta^{i}}$$ and $$x=\lim_{n\to\infty}\frac{1}{f(n)}\sum_{i=1}^{n}\epsilon_i.$$
\end{theorem}The following corollary is an immediate consequence of Theorem \ref{slow growth theorem}.
\begin{corollary}
Let $\beta\in(1,\frac{1+\sqrt{5}}{2}).$ Then for every $x\in(-\frac{1}{\beta-1},\frac{1}{\beta-1})$ there exists $(\epsilon_i)\in\{-1,1\}^{\mathbb{N}}$ such that $$x=\sum_{i=1}^{\infty}\frac{\epsilon_{i}}{\beta^{i}}$$ and $$x=\lim_{n\to\infty}\frac{1}{n^{1/2}}\sum_{i=1}^{n}\epsilon_i.$$
\end{corollary}
\subsection{A family of overlapping self-affine sets and simultaneous expansions}
Let $\{S_{j}\}_{j=1}^{m}$ be a collection of contracting maps acting on $\mathbb{R}^{d}$. A result of Hutchinson \cite{Hut} states that there exists a unique non-empty compact set $\Lambda\subseteq \mathbb{R}^{d}$ such that $$\Lambda=\bigcup_{j=1}^{m}S_j(\Lambda).$$ We call $\Lambda$ the \emph{attractor associated to $\{S_j\}$}. Often one is interested in determining the topological properties of $\Lambda$. When the collection $\{S_j\}$ consists solely of similarities than the attractor $\Lambda$ is reasonably well understood. However, when the collection $\{S_j\}$ contains affine maps the situation is known to be much more complicated.

In this paper we focus on the following family of self-affine sets. Let $1<\beta_{1},\beta_{2},\beta_{3}\leq 2$ and $$S_{-1}(x,y)=\Big(\frac{x-1}{\beta_{1}},\frac{x-1}{\beta_{2}}\Big) \textrm{ and }S_{1}(x,y)=\Big(\frac{x+1}{\beta_{1}},\frac{x+1}{\beta_{3}}\Big).$$ For this collection of contractions we denote the associated attractor by $\Lambda_{\b_1,\b_2,\b_3}.$ In Figure \ref{figa} we include some examples.

\begin{figure}[h]
\includegraphics[width=12cm, height=5cm]{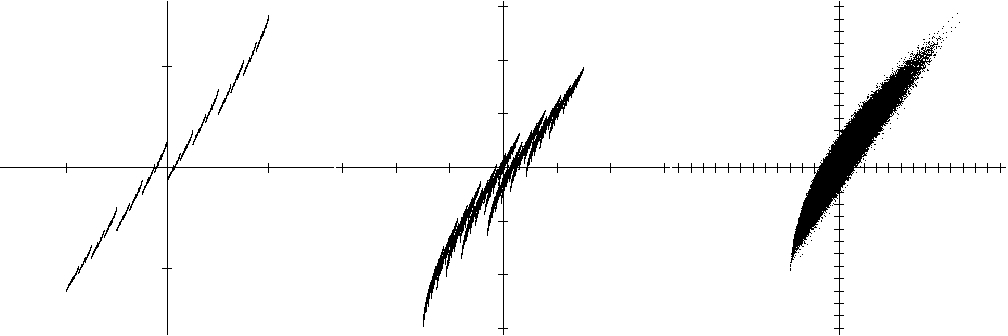}
\centering
\caption{A plot of $\Lambda_{2,1.81,1.66},\Lambda_{1.66,1.33,1.53},\Lambda_{1.2,1.11,1.05}$}
\label{figa}
\end{figure}

When $\beta_2=\b_3$ we denote $\Lambda_{\b_1,\b_2,\b_3}$ by $\Lambda_{\b_1,\b_2}.$ The case where $\beta_2=\beta_3$ was studied in \cite{DKK} and \cite{HS}. One problem the authors of these papers were particularly interested in was determining those pairs $(\beta_1,\beta_2)$ for which the attractor $\Lambda_{\b_1,\b_2}$ has non-empty interior. The best result in this direction is the following result due to Hare and Sidorov \cite{HS}.
\begin{theorem}
\label{KevinNik}
If $\beta_1\neq \beta_2$ and
\begin{equation}
\label{KevinNik condition}
\Big|\frac{\beta_{2}^{8}-\beta_{1}^{8}}{\beta_{2}^{7}-\beta_{1}^{7}}\Big|+\Big|\frac{\beta_{2}^7\beta_{1}^{7}(\beta_{2}-\beta_{1})}{\beta_{2}^{7}-\beta_{1}^{7}}\Big|\leq 2.
\end{equation} Then $\Lambda_{\b_1,\b_2}$ has non-empty interior and $(0,0)\in \Lambda^{\mathrm{o}}.$
\end{theorem}

Let $\pi$ denote the projection from $\mathbb{R}^{2}$ onto the $x$-axis. For each $x\in\pi(\Lambda_{\b_1,\b_2,\b_3})$ let $$\Lambda^{x}_{\b_1,\b_2,\b_3}:=\{y\in\mathbb{R}:(x,y)\in\Lambda_{\b_1,\b_2,\b_3}\}.$$ We call $\Lambda^{x}_{\b_1,\b_2,\b_3}$ the \emph{fibre of $x$}. Note that $\pi(\Lambda_{\b_1,\b_2,\b_3})=[\frac{-1}{\beta_1-1},\frac{1}{\beta_1-1}]$. The following statement is our main result for $\Lambda_{\b_1,\b_2,\b_3}.$

\begin{theorem}
\label{Affine theorem}
Let $\beta_{1}\in(1,\frac{1+\sqrt{5}}{2}).$ Then there exists $c=c(\beta_1)>0$ such that for all $\beta_{2},\beta_{3}\in(1,1+c)$ and $x\in(-\frac{1}{\beta_1-1},\frac{1}{\beta_1-1})$ the fibre $\Lambda_{\b_1,\b_2,\b_3}^{x}$ contains an interval. Moreover $\Lambda_{\beta_1,\beta_2,\b_3}$ has non-empty interior and $(0,0)\in\Lambda_{\beta_1,\beta_2,\b_3}^{o}$
\end{theorem}We emphasise that Theorem \ref{Affine theorem} covers the case where $\beta_2\neq \beta_3$. Our approach lends itself to explicit calculation and following our method one can obtain a lower bound for the value $c$ appearing in Theorem \ref{Affine theorem}. We include some explicit calculations in Section \ref{Section Explicit calculations}.

Note that for any $\beta_1$ sufficiently close to $\frac{1+\sqrt{5}}{2}$ the set of $\beta_2\in(1,2)$ satisfying \eqref{KevinNik condition} is empty. Consequently Theorem \ref{Affine theorem} provides new examples of $\beta_1,\beta_2$ for which $\Lambda_{\beta_1,\beta_2}^{o}$ is non-empty. Theorem \ref{Affine theorem} is also optimal in the following sense. For any $\beta_1\in [\frac{1+\sqrt{5}}{2},2)$ and $\beta_{2},\beta_{3}\in(1,2),$ there exists $x\in (-\frac{1}{\beta_1-1},\frac{1}{\beta_1-1})$ such that the fibre $\Lambda^{x}_{\b_1,\b_2,\b_3}$ is countable and therefore does not contain an interval. We explain why this is the case in Section \ref{remarks}.

It is natural to ask whether the property $\Lambda_{\b_1,\b_2,\b_3}^{x}$ contains an interval for every $x\in  (-\frac{1}{\beta_1-1},\frac{1}{\beta_1-1})$ is stronger than the property $\Lambda_{\beta_1,\beta_2,\b_3}^{o}\neq \emptyset$. This is in fact the case and is a consequence of the following proposition.

\begin{proposition}
$\Lambda_{\beta_1,\beta_2,\b_3}^{o}\neq \emptyset$ if and only if $\{x: \Lambda_{\b_1,\b_2,\b_3}^{x}\textrm{ contains an interval}\}$ contains an open dense subset of $[\frac{-1}{\beta_1-1},\frac{1}{\b_1-1}]$.
\end{proposition}
\begin{proof}
Let us start by introducing some notation. Let $F=\{x: \Lambda_{\b_1,\b_2,\b_3}^{x}\textrm{ contains an interval}\}$. Suppose $\Lambda_{\beta_1,\beta_2,\b_3}^{o}\neq \emptyset.$ Then there exists $I$ and $J$ two nontrivial open intervals such that $I\times J\subseteq \Lambda_{\beta_1,\beta_2,\beta_3}.$  Let $\phi_{-1}(x)=\frac{x-1}{\beta_1}$ and $\phi_{1}(x)=\frac{x+1}{\beta_1}.$ Since $S_{-1}(I\times J)$ is an open rectangle contained in $\Lambda_{\b_1,\b_2,\b_3},$ it follows that $\phi_{-1}(I)\subseteq F.$ Similarly $\phi_1(I)\subseteq F.$ Repeating this argument, it follows that all images of $I$ under finite concatenations of $\phi_{-1}$ and $\phi_1$ are contained in $F.$ The union of these images of $I$ is an open dense subset of $[\frac{-1}{\b_1-1},\frac{1}{\b_1-1}].$ It follows that $F$ contains an open dense subset of $[\frac{-1}{\beta_1-1},\frac{1}{\b_1-1}]$.

It remains to prove the leftwards implication. We start by partitioning the set $F$. Given $(a,b,c,d)\in\mathbb{Z}^4$ let $$F_{a,b,c,d}:=\Big\{x: \Big[\frac{a}{b},\frac{c}{d}\Big]\subseteq \Lambda_{b_1,b_2,b_3}^x\Big\}.$$ Importantly we have $$F=\bigcup_{(a,b,c,d)\in\mathbb{Z}^4,\,a/b< c/d}F_{a,b,c,d}.$$ Suppose $F_{a,b,c,d}$ is nowhere dense for all $(a,b,c,d)\in\mathbb{Z}^4.$ Since  $F$ contains an open dense set its complement is a nowhere dense set. It follows that $[\frac{-1}{\beta_1-1},\frac{1}{\beta_1-1}]$ is the countable union of nowhere dense sets. By the Baire category theorem this is not possible. Therefore there must exist $(a',b',c',d')\in\mathbb{Z}^4$ such that $a'/b'<c'/d'$ and $F_{a',b',c',d'}$ is dense in some non trivial interval $I'$. Since $\Lambda_{\b_1,\b_2,\b_3}$ is closed it follows that $$I' \times [a'/b',c'/d']\subseteq \Lambda_{\b_1,\b_2,\b_3}$$ and  $\Lambda_{\b_1,\b_2,\b_3}$ has non-empty interior.
\end{proof}
Interestingly computer simulations suggest that there exist examples where $\Lambda_{\beta_1,\beta_2,\b_3}^{o}\neq \emptyset$ yet $\{x:\Lambda_{\beta_1,\beta_2,\beta_3}^x\textrm{ is a singleton} \}$ is infinite and even has positive Hausdorff dimension. See Figure \ref{figb} for such an example.

\begin{figure}[h]
\includegraphics[width=7cm, height=7cm]{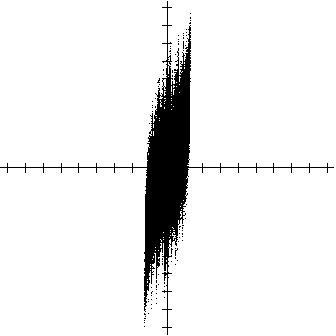}
\centering
\caption{A plot of $\Lambda_{1.8,1.05}$. For this choice of $\beta_1$ and $\beta_2$ it can be shown that $\{x:\Lambda_{\beta_1,\beta_2}^x\textrm{ is a singleton} \}$ has positive Hausdorff dimension.}
\label{figb}
\end{figure}

In \cite{Gun}, in addition to the notion of a hybrid expansion, G\"{u}nt\"{u}rk introduced the notion of a simultaneous expansion. These are defined as follows. Given $x\in [\frac{-1}{\beta_1-1},\frac{1}{\beta_1-1}]$ and $\beta_1,\beta_2\in(1,2),$ we say that a sequence $(\epsilon_i)\in\{-1,1\}^{\mathbb{N}}$ is a simultaneous $(\beta_{1},\beta_2)$ expansion of $x$ if $$x=\sum_{i=1}^{\infty}\frac{\epsilon_i}{\beta_1^i}=\sum_{i=1}^{\infty}\frac{\epsilon_i}{\beta_2^i}.$$ These expansions relate to our self-affine set via the following observation. If $\beta_2=\beta_3$ then $$\Lambda_{\b_1,\b_2}=\Big\{\Big(\sum_{i=1}^{\infty}\frac{\epsilon_i}{\beta_1^i},\sum_{i=1}^{\infty}\frac{\epsilon_i}{\beta_2^i}\Big): (\epsilon_i)\in\{-1,1\}^\mathbb{N}\Big\}.$$ Therefore $$\Big\{(x,x):x \textrm{ has a simultaneous }(\beta_1,\beta_2)\textrm{ expansion}\Big\}=\Lambda_{\b_1,\b_2} \cap \{(x,x):x\in\mathbb{R}\}.$$ In \cite{Gun} it was asserted by G\"{u}nt\"{u}rk that there exists $C>0,$ such that for $1<\beta_1<\beta_2<1+C,$ there exists $c=c(\beta_1,\beta_2)>0$ such that every $x\in (-c,c)$ has a simultaneous $(\beta_1,\beta_2)$ expansion. Note that the existence of $C>0$ satisfying the above follows if one can show that for $1<\beta_1<\beta_2<1+C$ the attractor $\Lambda_{\b_1,\b_2}$ contains $(0,0)$ in its interior. Using this observation G\"{u}nt\"{u}rk's assertion was proved to be correct in \cite{DKK}. The largest parameter space for which it is known that $(0,0)\in\Lambda_{\b_1,\b_2}^{o},$ and consequently that any $x$ sufficiently close to zero has a simultaneous $(\beta_1,\beta_2)$ expansion, is that stated in Theorem \ref{KevinNik}. Our contribution in this direction is the following theorem that follows as an immediate consequence of Theorem \ref{Affine theorem} by taking $\beta_2=\beta_3$.

\begin{theorem}
\label{simultaneous theorem}
Let $\beta_1\in(1,\frac{1+\sqrt{5}}{2})$. Then there exists $C=C(\beta_1)>0$ such that if $\beta_2\in(1,1+C),$ then every $x$ sufficiently small has a simultaneous $(\beta_1,\beta_2)$-expansion.
\end{theorem}
Before moving onto our proofs we say a few words about the methods used in this paper and compare them with those used in \cite{DKK} and \cite{HS}. In these papers the authors show that $(0,0)\in\Lambda_{\b_1,\b_2}^{o}$ by constructing a polynomial $P(x)=x^n+b_{n-1}x^{n-1}+b_1x+b_0$ which satisfies:
\begin{enumerate}
  \item $P(\beta_1)=P(\beta_2)=0$
  \item $\sum_{j=0}^{n-1}|b_j|\leq 2$
  \item $b_1=0$
  \item $b_0\neq 0.$
\end{enumerate}Once the existence of this polynomial is established, one can devise an algorithm which can be applied to any $x_1,x_2$ sufficiently small, this algorithm then yields an $(\epsilon_i)\in\{-1,1\}^{\mathbb{N}}$ such that $(x_1,x_2)=(\sum_{i=1}^{\infty}\epsilon_i \beta_1^{-i},\sum_{i=1}^{\infty}\epsilon_i \beta_2^{-i}).$

This approach is somewhat unsatisfactory. The existence of the polynomial and the algorithm used to construct the $(\epsilon_i)$ provide little intuition as to why $(0,0)$ should be in the interior of $\Lambda_{\b_1,\b_2}$. Our approach, as well as allowing for different rates of contraction in the vertical direction, is more intuitive and explicitly constructs the interval appearing in each fibre of $\Lambda_{\beta_1,\b_2,\b_3}$.

The rest of this paper is arranged as follows. In Section \ref{Section3} we recall and prove some technical results that are required to prove our theorems. In Section \ref{digit frequencies} we prove our theorems relating to digit frequencies. In Section \ref{affine section} we prove Theorem \ref{Affine theorem}. In Section \ref{Section Explicit calculations} we include an example where we explicitly calculate some of the parameters appearing in our theorems. In Section \ref{remarks} we include some general discussion and pose some questions.

\section{Preliminaries}
\label{Section3}
In this section we prove some useful technical results and recall some background material.  Let us start by introducing the maps $T_{-1}(x)=\beta x+1$, $T_0(x)=\beta x$ and $T_{1}(x)=\beta x -1$. Given an $x\in I_{\beta}$ we let $$\Sigma_{\beta}(x):=\Big\{(\epsilon_i)\in\{0,1\}^{\mathbb{N}}:\sum_{i=1}^{\infty}\frac{\epsilon_i}{\beta^i}=x\Big\}$$ and $$\Omega_{\beta}(x):=\Big\{(a_i)\in\{T_0,T_1\}^{\mathbb{N}}:(a_n\circ\cdots \circ a_1)(x)\in I_{\beta} \textrm{ for all }n\in\mathbb{N}\Big\}.$$ Similarly, given $x\in \widetilde{I}_{\beta}:=[\frac{-1}{\beta-1},\frac{1}{\beta-1}]$ let $$\widetilde{\Sigma}_{\beta}(x):=\Big\{(\epsilon_i)\in\{-1,1\}^{\mathbb{N}}:\sum_{i=1}^{\infty}\frac{\epsilon_i}{\beta^i}=x\Big\}$$ and $$\widetilde{\Omega}_{\beta}(x):=\Big\{(a_i)\in\{T_{-1},T_1\}^{\mathbb{N}}:(a_n\circ\cdots \circ a_1)(x)\in \widetilde{I}_{\beta} \textrm{ for all }n\in\mathbb{N}\Big\}.$$  The dynamical interpretation of $\beta$-expansions is best seen through the following result.

\begin{lemma}
\label{Bijection lemma}
For any $x\in I_{\beta}$($x\in\widetilde{I}_{\beta}$) we have $\textrm{Card }\Sigma_{\beta}(x)=\textrm{Card }\Omega_{\beta}(x)$($\textrm{Card }\widetilde{\Sigma}_{\beta}(x)=\textrm{Card }\widetilde{\Omega}_{\beta}(x)).$ Moreover, the map which sends $(\epsilon_i)$ to $(T_{\epsilon_i})$ is a bijection between $\Sigma_{\beta}(x)$ and $\Omega_{\beta}(x)$($\widetilde{\Sigma}_{\beta}(x)$ and $\widetilde{\Omega}_{\beta}(x)$).
\end{lemma}Lemma \ref{Bijection lemma} was originally proved in \cite{BakG} for an arbitrary digit set of the form $\{0,\ldots,m\}$. The proof easily extends to the digit set $\{-1,1\}$.

Lemma \ref{Bijection lemma} allows us to reinterpret problems from $\beta$-expansions in terms of the allowable trajectories that can occur within a dynamical system. In Figure \ref{fig1} we include a graph of $T_{0}$ and $T_{1}$ acting on $I_{\beta}$. One can see from this picture, or check by hand, that if $x\in [\frac{1}{\beta},\frac{1}{\beta(\beta-1)}]$ then both $T_0$ and $T_1$ map $x$ into $I_{\beta}$. Therefore, by Lemma \ref{Bijection lemma}, this $x$ has at least two $\beta$-expansions. More generally, if there exists a sequence of $T_0$'s and $T_1$'s that map $x$ into $[\frac{1}{\beta},\frac{1}{\beta(\beta-1)}],$ then $x$ has at least two $\beta$-expansions.

The interval $[\frac{1}{\beta},\frac{1}{\beta(\beta-1)}]$ is clearly important when it comes to studying  $\Sigma_{\beta}(x)$ and $\Omega_{\beta}(x)$. In what follows we let $$\mathcal{S}_{\beta}:=\Big[\frac{1}{\beta},\frac{1}{\beta(\beta-1)}\Big].$$






\begin{figure}
\centering
\begin{tikzpicture}[x=2.1,y=2.1]
\path[draw](0,10) -- (0,160) -- (150,160) -- (150,10) -- (0,10);
\path[draw][thick](0,10) -- (100,160);
\path[draw][thick](50,10) -- (150,160);
\path[draw](50,8) -- (50,12);
\path[draw](100,8) -- (100,12);
\draw (0,10) node[below] {$0$};
\draw (50,7) node[below] {$\frac{1}{\beta}$};
\draw (100,7) node[below] {$\frac{1}{\beta(\beta-1)}$};
\draw (150,10) node[below] {$\frac{1}{\beta-1}$};
\end{tikzpicture}
\caption{The overlapping graphs of $T_0$ and $T_1$.}
 \label{fig1}

\end{figure}
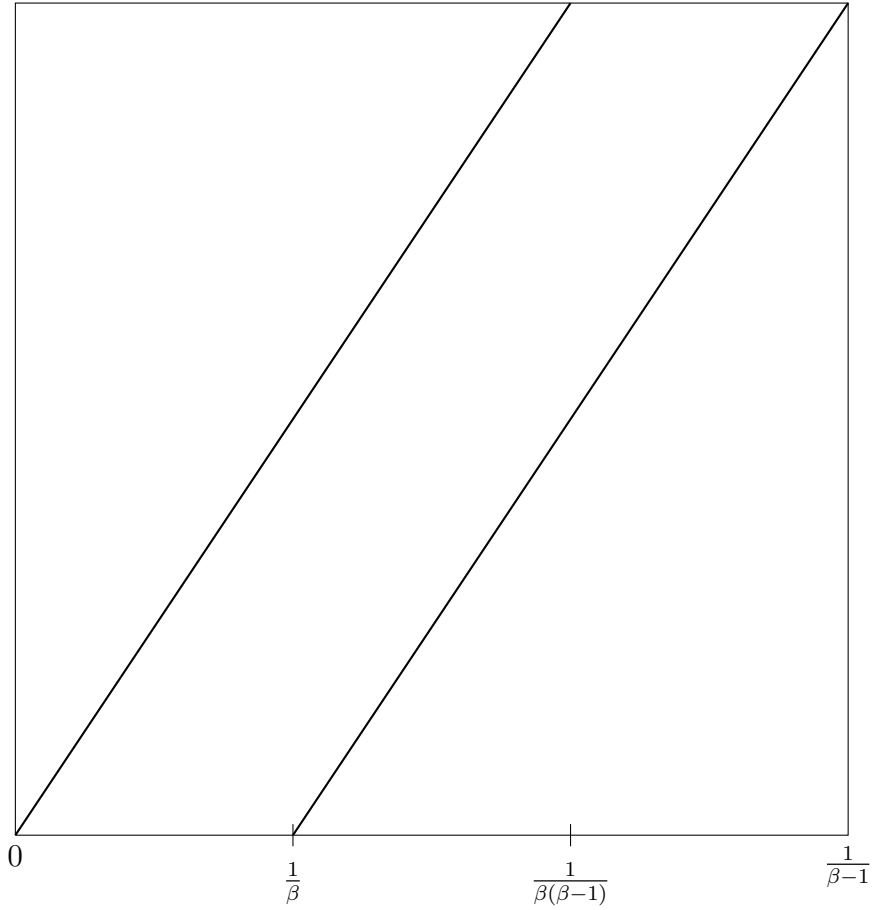

Another particularly useful interval for studying $\beta$-expansions is $$\mathcal{O}_{\beta}:=\Big[\frac{1}{\beta^2-1},\frac{\beta}{\beta^2-1}\Big].$$ The analogues of $\mathcal{S}_{\beta}$ and $\mathcal{O}_{\beta}$ for the digit set $\{-1,1\}$ are $$\widetilde{\mathcal{S}}_{\beta}:=\Big[\frac{\beta-2}{\beta(\beta-1)},\frac{2-\beta}{\beta(\beta-1)}\Big] \textrm{ and } \widetilde{\mathcal{O}}_{\beta}:=\Big[\frac{1-\beta}{\beta^2-1},\frac{\beta-1}{\beta^2-1}\Big].$$ The intervals $\mathcal{O}_{\beta}$ and $\widetilde{\mathcal{O}}_{\beta}$ are important because of the following lemma.
\begin{lemma}
\label{basic lemma}
For any $\beta\in(1,2)$ we have
\begin{equation}
\label{dumb1}
T_{0}\Big(\frac{1}{\beta^2-1}\Big)=\frac{\beta}{\beta^2-1}\textrm{ and }T_1\Big(\frac{\beta}{\beta^2-1}\Big)=\frac{1}{\beta^2-1},
\end{equation}
and
\begin{equation}
\label{dumb2}
T_{-1}\Big(\frac{1-\beta}{\beta^2-1}\Big)=\frac{\beta-1}{\beta^2-1}\textrm{ and }T_1\Big(\frac{\beta-1}{\beta^2-1}\Big)=\frac{1-\beta}{\beta^2-1}.
 \end{equation}Moreover, for any $x\in(0,\frac{1}{\beta-1})$ ($x\in(\frac{-1}{\beta-1},\frac{1}{\beta-1})$) there exists a sequence of $T_{0}$'s or $T_1$'s($T_{-1}$'s or $T_1$'s) that map $x$ into $\mathcal{O}_{\beta}$($\widetilde{\mathcal{O}}_{\beta}$).
\end{lemma}
\begin{proof}
Verifying \eqref{dumb1} and \eqref{dumb2} is a simple calculation. These equations tell us that it is not possible for an $x$ to be mapped over $\mathcal{O}_{\beta}$ or $\widetilde{\mathcal{O}}_{\beta}$ via an application of one of our maps. Note that the endpoints of the interval $I_{\beta}$ are the fixed points of the maps $T_0$ and $T_1$. Similarly the endpoints of the interval $\widetilde{I}_{\beta}$ are the fixed points of the maps $T_{-1}$ and $T_1$. Combining these observations with the expansivity of our maps implies the second half of our lemma.
\end{proof}

Several of our theorems will rely on the following proposition. Loosely speaking, it states that for $\beta\in(1,\frac{1+\sqrt{5}}{2})$, for any $x\in \mathcal{O}_{\beta}$ ($x\in \widetilde{\mathcal{O}}_{\beta})$ there exists a method of generating expansions of $x$ such that we have a lot of control over the digits that appear. Before we state this result it is useful to introduce some notation.

In what follows we let $\{T_0,T_1\}^{*}=\cup _{n=0}^{\infty}\{T_0,T_1\}^{n}$. Given $\omega=(\omega_1,\ldots,\omega_n)\in\{T_0,T_1\}^{*}$ let $\omega(x)=(\omega_n\circ\cdots\circ\omega_1)(x)$. We let $|\omega|$ denote the length of $\omega$. We also let $$|\omega|_0=\#\{1\leq i \leq |\omega|:\omega_i=T_0\}$$ and $$|\omega|_1=\#\{1\leq i \leq |\omega|:\omega_i=T_1\}.$$ For a finite word $\omega\in\{T_{0},T_{1}\}^{*}$ we denote by $\omega^k$ its $k$-fold concatenation with itself and by $\omega^{\infty}$ the infinite sequence obtained by concatenating $\omega$ indefinitely. The above notions translate over in the obvious way to sequences of maps whose components are from the set $\{T_{-1},T_1\}$. We also define $|\cdot|_{-1}$ in the obvious way.

\begin{proposition}
\label{important prop}
Let $\beta\in(1,\frac{1+\sqrt{5}}{2}).$ There exist $n(\beta)\in\mathbb{N}$ such that if $x\in \mathcal{O}_\beta$($x\in \widetilde{\mathcal{O}}_{\beta}$) then there exists $\omega^0,\omega^1\in \{T_0,T_1\}^{*}$($\omega^{-1},\omega^1\in \{T_{-1},T_1\}^{*}$) satisfying the following:
\begin{itemize}
\item $|\omega^0|\leq n(\beta)$ and $|\omega^1|\leq n(\beta)$\,($|\omega^{-1}|\leq n(\beta)$ and $|\omega^1|\leq n(\beta)$).
  \item $\omega^{0}(x)\in \mathcal{O}_{\beta}$ and $\omega^{1}(x)\in \mathcal{O}_\beta$\,($\omega^{-1}(x)\in \widetilde{\mathcal{O}}_{\beta}$ and $\omega^{1}(x)\in \widetilde{\mathcal{O}}_\beta$)
  \item $|\omega^{0}|_{0}>|\omega^{0}|_{1}$\,($|\omega^{-1}|_{-1}>|\omega^{-1}|_{1}$)
  \item $|\omega^{1}|_{1}>|\omega^{1}|_{0}$\,($|\omega^{1}|_{1}>|\omega^{1}|_{-1}$)
\end{itemize}
\end{proposition}

Let us take the opportunity to emphasise that $\omega^{-1}$ is not an inverse map.

We will only give a proof of Proposition \ref{important prop} for the digit set $\{0,1\}$. The case where the digit set is $\{-1,1\}$ is dealt with similarly. Before giving a proof of Proposition \ref{important prop} for the digit set $\{0,1\}$ it is useful to define two more intervals and state some basic facts. For any $\beta\in(1,2)$ let:
$$\mathcal{I}_{\beta}:=\Big[\frac{1}{2}\Big(\frac{1}{\beta}+\frac{1}{\beta^2-1}\Big),\frac{1}{2}\Big(\frac{\beta}{\beta^2-1}+\frac{1}{\beta(\beta-1)}\Big)\Big],$$
and
\begin{align*}
\mathcal{J}_{\beta}:&=\Big[T_{1}\Big(\frac{1}{2}\Big(\frac{1}{\beta}+\frac{1}{\beta^2-1}\Big)\Big),T_{0}\Big(\frac{1}{2}\Big(\frac{\beta}{\beta^2-1}+\frac{1}{\beta(\beta-1)}\Big)\Big)\Big]\\
&=\Big[\frac{\beta}{2}\Big(\frac{1}{\beta}+\frac{1}{\beta^2-1}\Big)-1,\frac{\beta}{2}\Big(\frac{\beta}{\beta^2-1}+\frac{1}{\beta(\beta-1)}\Big)\Big].\\
\end{align*} For any $\beta\in(1,2)$ these intervals are well defined and nontrivial. Note that the left endpoint of the interval of $\mathcal{I}_{\beta}$ is the midpoint of the left endpoints of $\mathcal{S}_{\beta}$ and $\mathcal{O}_{\beta},$ and the right endpoint of $\mathcal{I}_{\beta}$ is the midpoint of the right endpoints of $\mathcal{S}_{\beta}$ and $\mathcal{O}_{\beta}$

\begin{lemma}
\label{basic lemma2}
For any $\beta\in(1,\frac{1+\sqrt{5}}{2})$ we have $\mathcal{O}_{\beta}\subsetneq \mathcal{I}_{\beta}\subsetneq \mathcal{S}_{\beta}$ and $\mathcal{J}_{\beta}\subseteq (0,\frac{1}{\beta-1})$.
\end{lemma}
\begin{proof}
Lemma \ref{basic lemma2} will follow if we can prove that $$\frac{1}{\beta}<\frac{1}{2}\Big(\frac{1}{\beta}+\frac{1}{\beta^2-1}\Big)<\frac{1}{\beta^2-1}$$ and $$\frac{\beta}{\beta^2-1}<\frac{1}{2}\Big(\frac{\beta}{\beta^2-1}+\frac{1}{\beta(\beta-1)}\Big)<\frac{1}{\beta(\beta-1)}$$ for $\beta\in(1,\frac{1+\sqrt{5}}{2})$. Verifying these inequalities is a simple calculation and is omitted.
\end{proof}

\begin{lemma}
\label{basic lemma 2}
Let $\beta\in(1,\frac{1+\sqrt{5}}{2}).$ There exists $n_1(\beta)\in\mathbb{N}$ such that:
 \begin{itemize}
   \item If $$x\in\Big[\frac{\beta}{2}\Big(\frac{1}{\beta}+\frac{1}{\beta^2-1}\Big)-1,\frac{1}{\beta^2-1}\Big]$$ then $T_0^i(x)\in\mathcal{O}_{\beta}$ for some $1\leq i \leq n_{1}(\beta).$
   \item If $$x\in\Big[\frac{\beta}{\beta^2-1},\frac{\beta}{2}\Big(\frac{\beta}{\beta^2-1}+\frac{1}{\beta(\beta-1)}\Big)\Big]$$ then $T_1^i(x)\in\mathcal{O}_{\beta}$ for some $1\leq i \leq n_{1}(\beta).$
 \end{itemize}
\end{lemma}

\begin{proof}
We begin our proof by pointing out that the left endpoint of $I_{\beta}$ is the fixed point of $T_0$ and the right endpoint of $I_{\beta}$ is the fixed point of $T_{1}$. Moreover, the maps $T_{0}$ and $T_{1}$ expand distances from their respective fixed points in the following way:
\begin{equation}
\label{expansion equation}
T_{0}(x)-0=\beta(x -0) \textrm{ and } T_{1}(x)-\frac{1}{\beta-1}=\beta\big(x-\frac{1}{\beta-1}\Big).
\end{equation} Let us fix $$x\in\Big[\frac{\beta}{2}\Big(\frac{1}{\beta}+\frac{1}{\beta^2-1}\Big)-1,\frac{1}{\beta^2-1}\Big].$$ The second case is dealt with similarly. Lemma \ref{basic lemma2} guarantees $$\frac{\beta}{2}\Big(\frac{1}{\beta}+\frac{1}{\beta^2-1}\Big)-1>0.$$ Let $n_1(\beta)\in\mathbb{N}$ be the unique natural number which satisfies $$\beta^{n_1(\b)-1}\Big(\frac{\beta}{2}\Big(\frac{1}{\beta}+\frac{1}{\beta^2-1}\Big)-1)\leq \frac{1}{\beta^2-1}<\beta^{n_1(\b)}\Big(\frac{\beta}{2}\Big(\frac{1}{\beta}+\frac{1}{\beta^2-1}\Big)-1).$$ Then by \eqref{expansion equation}, the monotonicity of $T_0,$ and the first part of Lemma \ref{basic lemma}, there must exist $1\leq i\leq n_{1}(\beta)$ such that $T_{0}^i(x)\in\mathcal{O}_{\beta}$.
\end{proof}
Equipped with Lemma \ref{basic lemma 2} we are now in a position to prove Proposition \ref{important prop}.

\begin{proof}[Proof of Proposition \ref{important prop}.]
Let $\beta\in(1,\frac{1+\sqrt{5}}{2})$ and $x\in \mathcal{O}_\beta$. We only show how to construct $\omega^0$. The construction of $\omega^1$ follows from an analogous argument. Alternatively, one could consider $x'=\frac{1}{\beta-1}-x.$ It can be shown that if we took the corresponding $\omega^0$ for $x'$ and replaced every occurrence of $T_0$ with $T_1$ and $T_1$ with $T_0,$ then the resulting sequence would have the desired properties of a $\omega^1$ for our original $x$.

Let us start by considering the image of $x$ under $T_{1}(x)$. By Lemma  \ref{basic lemma} we know that $T_{1}(x)\in[\frac{\beta}{\beta^2-1}-1,\frac{1}{\beta^2-1}]$. There are two cases to consider, either $T_{1}(x)\notin \mathcal{I}_{\beta}$ or $T_{1}(x)\in \mathcal{I}_{\beta}$. We start with the first case.
\\

\noindent \textbf{Case 1. }Suppose $T_{1}(x)\notin \mathcal{I}_{\beta},$ then $$T_1(x)<\frac{1}{2}\Big(\frac{1}{\beta}+\frac{1}{\beta^2-1}\Big)$$ and consequently
$$x<\frac{\beta}{\beta^2-1}-\delta(\beta)$$
where
\begin{align*}
\delta(\beta)&:=\frac{\beta}{\beta^2-1}-T_{1}^{-1}\Big(\frac{1}{2}\Big(\frac{1}{\beta}+\frac{1}{\beta^2-1}\Big)\Big) \\
&=\frac{\beta}{\beta^2-1}-\frac{1}{\beta}-\frac{1}{2\beta}\Big(\frac{1}{\beta}+\frac{1}{\beta^2-1}\Big)>0.
\end{align*}
Importantly $\delta=\delta(\beta)$ only depends upon $\beta$.

We repeatedly apply $T_{0}$ to $T_{1}(x)$ until $(T_{0}^{i_{1}}\circ T_{1})(x)\in \mathcal{O}_{\beta}.$ This is permissible by Lemma \ref{basic lemma}. It is a consequence of Lemma \ref{basic lemma 2} that $ i_{1} \leq n_1(\beta)$ for some $n_1(\beta)$ that only depends upon $\beta$. If $i_{1}>1$ then we stop and take $\omega^{0}= (T_1,(T_{0})^{i_{1}}).$

If $i_1=1$ then $$(T_{0}\circ T_{1})(x)\in \mathcal{O}_{\beta}$$ and
\begin{equation}
\label{escape}
(T_{0}\circ T_{1})(x)<\frac{\beta}{\beta^2-1}-\beta^2\delta(\beta).
 \end{equation}Equation \eqref{escape} follows because $\frac{\beta}{\beta^2-1}$ is the unique fixed point of $T_{0}\circ T_{1}$ and this map expands distances by a factor $\beta^2$. We now repeat our first step with $x$ replaced by $(T_{0}\circ T_{1})(x)$, i.e. we consider $(T_{1}\circ T_{0}\circ T_{1})(x)$ and apply $T_{0}$ until $(T_{0}^{i_{2}}\circ T_{1}\circ T_{0}\circ T_{1})(x)\in \mathcal{O}_{\beta}.$ If $i_2>1$ we stop and take $\omega^0=(T_1,T_0,T_1,(T_{0})^{i_2})$.

If $i_2=1$ then $$(T_{0}\circ T_{1}\circ T_{0}\circ T_{1})(x)\in\mathcal{O}_{\beta}$$ and $$(T_{0}\circ T_{1}\circ T_{0}\circ T_{1})(x)<\frac{\beta}{\beta^2-1}-\beta^4\delta(\beta).$$ One can then repeat our first step with $x$ replaced by $(T_0\circ T_1 \circ T_{0}\circ T_{1})(x)$ and so on, each time obtaining a value for $i_j$ and stopping as soon as $i_j$ is strictly larger then $1$. Suppose we repeated this process $k$ times and each time our value of $i_j$ was $1.$ Then we would have
\begin{equation*}
(\overbrace{(T_{0}\circ T_{1})\circ \cdots \circ(T_0\circ T_{1})}^{k\textrm{-times}})(x)\in\mathcal{O}_{\beta}.
\end{equation*} In which case
\begin{equation}
\label{Distance from endpoint}
\beta^{2k}\delta(\beta)<\frac{\beta}{\beta^2-1}-(\overbrace{(T_{0}\circ T_{1})\circ \cdots \circ (T_0\circ T_{1})}^{k\textrm{-times}})(x)\leq \frac{\beta}{\beta^2-1}-\frac{1}{\beta^{2}-1}.
\end{equation}Let $n_2(\beta)\in\mathbb{N}$ be the unique natural number satisfying
\begin{equation}
\label{n_{2} equation}
\beta^{2(n_{2}(\beta)-1)}\delta(\beta)\leq \frac{\beta}{\beta^2-1}-\frac{1}{\beta^{2}-1}<\beta^{2n_{2}(\beta)}\delta(\beta).
\end{equation}
By \eqref{Distance from endpoint} and \eqref{n_{2} equation} there must exists $k\leq n_2(\beta)$ such that $$\overbrace{(T_{0}\circ T_{1})\circ \cdots \circ(T_0\circ T_{1})}^{k\textrm{-times}}(x)\notin\mathcal{O}_{\beta}.$$ At which point $i_k>1$. We may take $\omega^0$ to be
\begin{equation}
\label{omega}
\omega^0=((T_{1},T_0)^{k-1}),T_1,T_{0}^{i_{k}}).
\end{equation}Note that $i_k\leq n_1(\beta)$ by Lemma \ref{basic lemma 2} and therefore $|\omega_{0}|\leq n_1(\beta)+2n_2(\beta)-1.$ This upper bound only depends upon $\beta$. The fact that $\omega^{0}(x)\in\mathcal{O}_{\beta}$ follows from our algorithm. Moreover, it is clear from inspection of \eqref{omega} that $|\omega^{0}|_0>|\omega^{0}|_1.$ Therefore $\omega^{0}$ satisfies each of the required properties.
\\

\noindent \textbf{Case 2. }If $T_1(x)\in \mathcal{I}_{\beta}$ then $$T_{1}(x)\in\Big[\frac{1}{2}\Big(\frac{1}{\beta}+\frac{1}{\beta^2-1}\Big),\frac{1}{\beta^{2}-1}\Big].$$ We consider $T_{1}^2(x)$ and repeatedly apply $T_0$ until $(T_0^{i_{1}}\circ T_{1}^2)(x)\in \mathcal{O}_{\beta}$. We cannot have $i_{1}=1,$ since by the monotonicity of our maps we would then have $$(T_{0}\circ T_{1})\Big(\frac{1}{\beta^{2}-1}\Big)\geq \frac{1}{\beta^2-1}.$$ Which is not possible since $T_{0}\circ T_{1}$ expands the distance from the fixed point $\frac{\beta}{\beta^2-1}$ by a factor $\beta^2$. By Lemma \ref{basic lemma 2} we must have $i_{1}\leq n_1(\beta)$. If $i_{1}>2$ then we may stop and take $\omega^0=(T_{1},T_{1},(T_0)^{i_{1}}).$

If $i_1=2$ then $$(T_{0}\circ T_0\circ T_1\circ T_{1})(x)\in \Big[\frac{1}{\beta^2-1},(T_{0}\circ T_0\circ T_{1}\circ T_1)\Big(\frac{\beta}{\beta^2-1}\Big)\Big].$$
But for any $\beta\in(1,2)$ it can be shown that $$(T_{0}\circ T_0\circ T_{1}\circ T_1)\Big(\frac{\beta}{\beta^2-1}\Big)=\frac{\beta^3+\beta^2-\beta^4}{\beta^2-1}<\frac{\beta}{\beta^2-1}.$$ Therefore if $i_1=2$ then $$(T_{0}\circ T_0\circ T_{1}\circ T_1)(x)\leq \frac{\beta}{\beta^2-1}-\delta'(\beta),$$ where $$\delta'(\beta):=\frac{\beta}{\beta^2-1}-\frac{\beta^3+\beta^2-\beta^4}{\beta^2-1}>0.$$ We are now in a position where we can replicate the arguments used in Case $1$. We apply $T_{1}$ to $(T_{0}\circ T_0\circ T_{1}\circ T_1)(x)$ and then repeatedly apply $T_0$ until the orbit returns to $\mathcal{O}_\beta$. If we apply $T_0$ more than once we stop, if we apply $T_0$ only once then we repeat the previous step. The positivity of $\delta'(\beta)$ implies that the number of times an orbit can immediately return to $\mathcal{O}_\beta$ is bounded above by some parameter only depending upon $\beta$. We also know by Lemma \ref{basic lemma 2} that the number of iterations of $T_0$ required to map our orbit back into $\mathcal{O}_{\beta}$ is bounded above by some constant that only depends upon $\beta$. These two remarks imply the existence of the required $\omega^{0}$ and $n(\beta)$.

\end{proof}

Proposition \ref{important prop} allows us to effectively handle the parameter space $(1,\frac{1+\sqrt{5}}{2})$. To prove results within the interval $[\frac{1+\sqrt{5}}{2},2),$ we need to recall some background results from unique expansions. Given $\beta\in(1,2),$ let $$\mathcal{U}_{\beta}:=\Big\{x\in\Big[0,\frac{1}{\beta-1}\Big]:x \textrm{ has a unique }\beta\textrm{-expansion}\Big\}$$and
 $$\widetilde{\mathcal{U}}_{\beta}:=\Big\{(\epsilon_i)\in\{0,1\}^{\mathbb{N}}:\sum_{i=1}^{\infty}\frac{\epsilon_i}{\beta^i}\in\mathcal{U}_{\beta}\Big\}.$$
We call $\mathcal{U}_{\beta}$ the \emph{univoque set} and $\widetilde{\mathcal{U}}_{\beta}$ the \emph{univoque sequences}. By definition there is a bijection between these two sets. The study of these sets is classical. For more on these sets we refer the reader to \cite{BarBakKong, deVKom,KomKonLi} and the references therein.

A useful tool for studying univoque sequences is the lexicographic ordering. This is defined as follows. Given $(\epsilon_i),(\delta_i)\in\{0,1\}^{\mathbb{N}},$ we say that $(\epsilon_i)\prec (\delta_i)$ if $\epsilon_1<\delta_1,$ or if there exists $n\in\mathbb{N}$ such that $\epsilon_{n+1}<\delta_{n+1}$ and $\epsilon_i=\delta_i$ for $1\leq i \leq n$. One can also define $\preceq, \succ, \succeq$ in the natural way. We also let $\overline{\epsilon_i}=1-\epsilon_i$. When studying univoque sequences an important role is played by the \emph{quasi-greedy expansion} of $1$. This sequence is defined to be the lexicographically largest infinite $\beta$-expansion of $1$. We call a sequence infinite if it does not end in an infinite tail of zeros. Given $\beta\in(1,2)$ we denote the quasi-greedy expansion of $1$ by $\alpha(\beta)=(\alpha_i(\beta))$. The following characterisation of quasi-greedy expansions is due to Baiocchi and Komornik \cite{BaiKom}.

\begin{lemma}
\label{quasi greedy properties}
The map which sends $\beta$ to $\alpha(\beta)$ is a strictly increasing bijection from $(1,2]$ onto the set of sequences $(\alpha_i)\in\{0,1\}^{\mathbb{N}}$ which satisfy
\begin{equation}
\label{lexbound}
(\alpha_{n+i})\preceq (\alpha_i) \textrm{ whenever }\alpha_n=0.
\end{equation}
\end{lemma}

We remark that if $x\in \mathcal{U}_{\beta}$ and $x\notin\{0,\frac{1}{\beta-1}\},$ then $x$ is eventually mapped into $(\frac{2-\beta}{\beta-1},1).$ Moreover, it is a consequence of being in $\mathcal{U}_{\beta}$ that once $x$ is mapped into $(\frac{2-\beta}{\beta-1},1),$ it cannot be mapped outside of $(\frac{2-\beta}{\beta-1},1).$ Consequently the following sets can be thought of as attractors for $\mathcal{U}_{\beta}$ and $\widetilde{\mathcal{U}}_{\beta}$. Let
$$\mathcal{A}_{\beta}:=\Big\{x\in\Big(\frac{2-\beta}{\beta-1},1\Big):x \textrm{ has a unique }\beta\textrm{-expansion}\Big\}$$and
 $$\widetilde{\mathcal{A}}_{\beta}:=\Big\{(\epsilon_i)\in\{0,1\}^{\mathbb{N}}:\sum_{i=1}^{\infty}\frac{\epsilon_i}{\beta^i}\in\mathcal{A}_{\beta}\Big\}$$
 The above attractor observation and the following lemma are due to Glendinning and Sidorov \cite{GlenSid}.
\begin{lemma}
\label{lexicographic lemma}
$$\widetilde{\mathcal{A}}_{\beta}=\Big\{(\epsilon_i)\in\{0,1\}^{\mathbb{N}}:(\overline{\alpha_i(\beta)})\prec(\epsilon_{n+i})\prec(\alpha_i(\beta))\textrm{ for all }n\in\mathbb{N}\Big\}$$
\end{lemma} Lemma \ref{lexicographic lemma} demonstrates the importance of the sequences $\alpha(\beta)$ when studying univoque sequences. The following lemma is a consequence of Lemma \ref{quasi greedy properties} and Lemma \ref{lexicographic lemma}.

\begin{lemma}
\label{inclusion lemma}
$\widetilde{\mathcal{A}}_{\beta}\subseteq \widetilde{\mathcal{A}}_{\beta'}$ for $\beta<\beta'.$
\end{lemma}
To extend our frequency results to the parameter space $(\frac{1+\sqrt{5}}{2},\beta_{KL})$ it is instructive to recall some properties of the Thue-Morse sequence. There are various ways to define the Thue-Morse sequence, we choose the following way defined via an iterative reflection process. Let $\tau^{0}=0$ and define $\tau^{1}$ to be $\tau^0$ concatenated with $\overline{\tau^0}$, i.e, $\tau^1=\tau^0\overline{\tau^0}=01.$ We then define $\tau^2$ to be $\tau^2:=\tau^{1}\overline{\tau^{1}}.$ We continue this process inductively, given $\tau^k$ let $\tau^{k+1}=\tau^k\overline{\tau^k}.$ We can repeat this process indefinitely and in doing so we obtain an infinite limit sequence $\tau:=(\tau_i)_{i=0}^{\infty}$. This $\tau$ is the Thue-Morse sequence. The first few $\tau^k$ and the initial digits of $\tau$ are listed below:
$$\tau^{0}=0,\, \tau^{1}=01,\, \tau^2=0110,\, \tau^3=01101001$$
$$\tau=0110\, 1001\, 1001\, 0110\cdots.$$ For more on the Thue-Morse sequence we refer the reader to \cite{AllShall}. The significance of the Thue-Morse sequence within expansions in non-integer bases is that the Komornik-Loreti constant $\beta_{KL}\approx 1.787$, that is the smallest $\beta\in(1,2)$ such that $1$ has a unique $\beta$-expansion, is the unique solution to the equation $$1=\sum_{i=1}^{\infty}\frac{\tau_i}{\beta^{i}}.$$ For a proof of this fact see \cite{KomLor}. In \cite{AllCos} it was shown that $\beta_{KL}$ is transcendental.

Of particular importance to us are the sequences $$\upsilon^n=(\upsilon_i^n)_{i=1}^{\infty}:=(\tau_1,\ldots,\tau_{2^{n}-1},0)^{\infty}$$ and $$\kappa^n:=(T_{\tau_{0}},T_{\tau_1},\ldots, T_{\tau_{2^{n}-1}})\in\{T_0,T_1\}^{2^{n}}.$$ It can be shown that the sequences $\upsilon^n$ all satisfy \eqref{lexbound}. Therefore by Lemma \ref{quasi greedy properties} for each $n\in\mathbb{N}$ there exists $\beta_{n}\in (1,2)$ such that $\alpha(\beta_n)=\upsilon^{n}.$ It follows from the definitions that $\beta_{1}=\frac{1+\sqrt{5}}{2}$ and $\beta_n\nearrow \beta_{KL}.$ Moreover, for any $\beta\in [\beta_{n},\beta_{n+1})$ we have the following properties:
\begin{equation}
\label{switch1}
\pi_{\beta}((\tau^1)^{\infty})<\pi_{\beta}((\tau^2)^{\infty})<\cdots<\pi_{\beta}((\tau^n)^{\infty})\leq \frac{1}{\beta}< \pi_{\beta}((\tau^{n+1})^{\infty})
\end{equation}
\begin{equation}
\label{switch2}
\pi_{\beta}((\overline{\tau^{n+1}})^{\infty})<\frac{1}{\beta(\beta-1)}\leq \pi_{\beta}((\overline{\tau^n})^{\infty})<\cdots <\pi_{\beta}((\overline{\tau^2})^{\infty})<\pi_{\beta}((\overline{\tau^1})^{\infty})
\end{equation}and
\begin{equation}
\label{switch3}
\pi_{\beta}((\tau^{n+1})^{\infty})<\pi_{\beta}((\overline{\tau^{n+1}})^{\infty}).
\end{equation}Equations \eqref{switch1} and \eqref{switch2} are a consequence of the main result of \cite{AllSid}, in particular see Theorem $1.3$ and Proposition $2.16$ from this paper. Proving equation \eqref{switch3} holds is a straightforward calculation.

We also highlight the following facts which are a consequence of the Thue-Morse construction. For all $\beta\in(1,2)$
\begin{equation}
\label{fix equation}
\kappa^n(\pi_{\beta}((\tau^n)^\infty))=\pi_{\beta}((\tau^n)^\infty) \textrm{ and }\overline{\kappa^n}(\pi_{\beta}((\overline{\tau^n})^\infty))=\pi_{\beta}((\overline{\tau^n})^\infty)
\end{equation}
\begin{equation}
\label{flip equation}
\kappa^n(\pi_{\beta}((\tau^{n+1})^\infty))=\pi_{\beta}((\overline{\tau^{n+1}})^\infty) \textrm{ and }\overline{\kappa^n}(\pi_{\beta}((\overline{\tau^{n+1}})^\infty))=\pi_{\beta}((\tau^{n+1})^\infty).
\end{equation}
In \eqref{fix equation} and \eqref{flip equation} we have used $\overline{\kappa^n}$ to denote the sequence of maps obtained by replacing each $T_0$ in $\kappa^n$ with $T_1,$ and each $T_1$ in $\kappa^n$ with $T_0$. Observe that \eqref{fix equation} asserts that  $\pi_{\beta}((\tau^n)^\infty)$ and $\pi_{\beta}((\overline{\tau^n})^\infty)$ are the fixed points of $\kappa^n$ and $\overline{\kappa^n}$ respectively, and \eqref{flip equation} states that $\pi_{\beta}((\tau^{n+1})^\infty)$ and $\pi_{\beta}((\overline{\tau^{n+1}})^\infty)$ are mapped from one to the other by $\kappa^n$ and $\overline{\kappa^n}$ respectively. As we will see, these points will play a similar role to that played by the endpoints of $I_{\beta}$ and $\mathcal{O}_{\beta}$ within the parameter space $(1,\frac{1+\sqrt{5}}{2}).$

The following lemma is a consequence of the construction of the Thue-Morse sequence described above.

\begin{lemma}
\label{tau normal}
For all $n\geq 1$ we have $$\frac{\#\{0\leq i \leq |\tau^n|-1:\tau_i^{n}=0\}}{|\tau^n|}=\frac{1}{2}\textrm{ and }\frac{\#\{0\leq i \leq |\overline{\tau^n}|-1:\overline{\tau_i^{n}}=0\}}{|\overline{\tau^n}|}=\frac{1}{2}.$$ Consequently, $(\tau^n)^{\infty}$ and $(\overline{\tau^n})^{\infty}$ are simply normal for all $n\in\mathbb{N}$. Similarly, for all $n\geq 1$ we have $$\frac{|\kappa^n|_0}{|\kappa^n|}=\frac{1}{2} \textrm{ and } \frac{|\kappa^n|_1}{|\kappa^n|}=\frac{1}{2}.$$
\end{lemma}Lemma \ref{tau normal} implies that if $x$ can be mapped onto $\pi_{\beta}((\tau^n)^\infty)$ or $\pi_{\beta}((\overline{\tau^n})^\infty)$ then $x$ must have a simply normal $\beta$-expansion. This observation will be used in our proof of Theorem \ref{frequency theorem}.

\section{Proofs of our digit frequency statements}
\label{digit frequencies}

\subsection{Proofs for Theorem \ref{frequency theorem} and Theorem \ref{hybrid theorem}}
We start this section by proving a proposition that implies statement $3$ of Theorem \ref{frequency theorem}, and statement $1$ of Theorem \ref{frequency theorem} for the parameter space $(1,\frac{1+\sqrt{5}}{2}).$ This proposition will also allow us to immediately prove Theorem \ref{hybrid theorem}.

\begin{proposition}
\label{freq prop}
Let $\beta\in(1,\frac{1+\sqrt{5}}{2}).$ Then there exists $c=c(\beta)>0,$ such that for every $p\in[1/2-c,1/2+c]$ and $x\in(0,\frac{1}{\beta-1})$($x\in(-\frac{1}{\beta-1},\frac{1}{\beta-1})$), there exists $(\epsilon_i)\in\{0,1\}^{\mathbb{N}}$($\epsilon_i\in\{-1,1\}^{\mathbb{N}}$) such that $\sum_{i=1}^{\infty}\epsilon_i\beta^{-i}=x$ and $\textrm{freq}_0(\epsilon_i)=p$($\textrm{freq}_{-1}(\epsilon_i)=p$).
\end{proposition}

\begin{proof}
We only give a proof for the digit set $\{0,1\}$. The digit set $\{-1,1\}$ is dealt with similarly. Let us start by fixing $\beta\in(1,\frac{1+\sqrt{5}}{2})$ and $x\in(0,\frac{1}{\beta-1}).$ Lemma \ref{basic lemma} states that every $x\in(0,\frac{1}{\beta-1})$ is mapped into $\mathcal{O}_{\beta}$ by some finite sequence of maps. Since the frequency of zeros of a sequence is independent of any initial finite block, we may therefore assume without loss of generality that $x\in\mathcal{O}_{\beta}$.

Let $n(\beta)$ be as in Proposition \ref{important prop}. Consider an element $\omega\in\{T_0,T_1\}^{*}$ such that $|\omega|\leq n(\beta)$ and $|\omega|_0>|\omega|_1,$ then
\begin{equation}
\label{positive growth}
\frac{|\omega|_0}{|\omega|}\geq \frac{1}{|\omega|}\Big(\Big[\frac{|\omega|}{2}\Big]+1\Big)\geq \frac{1}{2}+\frac{1}{2n(\beta)}.
\end{equation}
The second inequality in \eqref{positive growth} is a consequence of  $|\omega|\leq n(\beta)$ and the following formula
\[ \frac{1}{|\omega|}\Big(\Big[\frac{|\omega|}{2}\Big]+1\Big) = \left\{ \begin{array}{ll}
         \frac{1}{2}+\frac{1}{2k} & \mbox{if $|\omega|=2k$  };\\
        \frac{1}{2}+\frac{1}{2(2k+1)} & \mbox{if $|\omega|=2k+1$}.\end{array} \right. \] Similarly, if $|\omega|\leq n(\beta)$ and $|\omega|_1>|\omega|_0,$ then
\begin{equation}
\label{negative growth}
\frac{|\omega|_0}{|\omega|}\leq \frac{1}{|\omega|}\Big(\Big[\frac{|\omega|}{2}\Big]-1\Big)\leq \frac{1}{2}-\frac{1}{2n(\beta)}.
\end{equation}
We now show that for any $$p\in\Big[\frac{1}{2}-\frac{1}{2n(\beta)},\frac{1}{2}+\frac{1}{2n(\beta)}\Big],$$ there exists a sequence $(\epsilon_i)$ such that $(\epsilon_i)$ is a $\beta$-expansion of $x$ and $\textrm{freq}_0(\epsilon_i)=p.$ We do this by constructing an algorithm which for any $p$ yields the desired sequence $(\epsilon_i)$. Our result will then follow by taking $c=(2n(\beta))^{-1}.$
\\

\noindent \textbf{Step $1.$ } If $p\in[\frac{1}{2}-\frac{1}{2n(\beta)},\frac{1}{2})$ then map $x$ to $\omega^{1}(x).$ If $p\in[\frac{1}{2},\frac{1}{2}+\frac{1}{2n(\beta)}]$ then map $x$ to $\omega^{0}(x).$ Here $\omega^0$ and $\omega^1$ are the sequences of transformations guaranteed by Proposition \ref{important prop}. Whichever of these maps we apply we call it $\lambda^{1}.$ Note that we trivially have the inequality $$ \big||\lambda^1|_0-p|\lambda^0|\big|\leq n(b).$$ We finish our first step by remarking that $\lambda^1(x)\in\mathcal{O}_{\beta}$ by Proposition \ref{important prop}.
\\

\noindent \textbf{Step $k+1$. } Suppose we have constructed $\lambda^k\in\{T_0,T_1\}^{*}$ such that $\lambda^{k}(x)\in \mathcal{O}_{\beta}$ and
\begin{equation}
\label{step k}
\big||\lambda^k|_0-p|\lambda^k|\big|\leq n(b).
\end{equation}We now show how to construct $\lambda^{k+1}$ satisfying \eqref{step k}. Either $|\lambda^k|_0\geq p|\lambda^k|$ or $|\lambda^k|_0< p|\lambda^k|.$ If $|\lambda^k|_0\geq p|\lambda^k|$ we take the map $\omega^1$ guaranteed by Proposition \ref{important prop} and apply it to $\lambda^{k}(x).$ We then let $\lambda^{k+1}=(\lambda^k,\omega^1)$ and observe that
\begin{align}
\label{low n}
|\lambda^{k+1}|_0 - p|\lambda^{k+1}|&=|\lambda^{k}|_{0}+|\omega^{1}|_0 - p(|\lambda^{k}|+|\omega^1|) \nonumber\\
&\geq (|\lambda^{k}|_0-p|\lambda^k|_0)+ |\omega^1|_0- p|\omega^1| \nonumber\\
&\geq - p|\omega^1|\nonumber\\
&\geq -n(\beta).
\end{align}
In our final inequality we used the fact that $1\leq |\omega^1|\leq n(\beta)$. Similarly, we have
\begin{align}
\label{high n}
|\lambda^{k+1}|_0 - p|\lambda^{k+1}|&=|\lambda^k|_{0}+|\omega^1|_0 - p(|\lambda^k|+|\omega^1|)\nonumber\\
&\leq (|\lambda^k|_0-p|\lambda^k|)+ |\omega^1|_0- p|\omega^1|\nonumber\\
&\leq n(\beta)+ |\omega^1|_0- p|\omega^1|&&(\textrm{By }\eqref{step k})\nonumber\\
& \leq n(\beta) + \Big(\frac{1}{2}-\frac{1}{2n(\beta)}\Big)|\omega^1|- p|\omega^1|&&(\textrm{By }\eqref{negative growth})\nonumber\\
&\leq n(\beta)&&(\textrm{Since }p\in \Big[\frac{1}{2}-\frac{1}{2n(\beta)},\frac{1}{2}+\frac{1}{2n(\beta)}\Big]).
\end{align}By \eqref{low n} and \eqref{high n} we have
\begin{equation}
\label{k+1 bound}
\big||\lambda^{k+1}|_0 - p|\lambda^{k+1}|\big|\leq n(\beta).
\end{equation}We also have $\lambda^{k+1}(x)\in\mathcal{O}_{\beta}$ by Proposition \ref{important prop}. If $|\lambda^k|_0< p|\lambda|_0$ we let $\lambda^{k+1}=(\lambda^k,\omega^0).$ One can then adapt the calculations done above to verify that \eqref{k+1 bound} still holds and $\lambda^{k+1}(x)\in\mathcal{O}_{\beta}$. This completes our inductive step.
\\

Clearly we can repeat step $k+1$ indefinitely. In doing so we obtain an infinite sequence $\lambda=(\lambda_i)\in\{T_{0},T_{1}\}^{\mathbb{N}}$. Since $\lambda^k(x)\in\mathcal{O}_{\beta}$ for each $k\in\mathbb{N},$ it follows that $\lambda\in\Omega_{\beta}(x).$ It remains to check that $\lambda$ has the correct frequency of maps. Lemma \ref{Bijection lemma} will then give us our desired element of $\Sigma_{\beta}(x)$.

For any $n\in\mathbb{N}$ consider the quantity $|(\lambda_i)_{i=1}^{n}|_0/n.$ For each $n$ there exists $k_n$ such that $|\lambda^{k_n}|\leq n < |\lambda^{k_{n}+1}|.$ By \eqref{step k} we have
\begin{equation}
\label{p upper}
\frac{|(\lambda_i)_{i=1}^{n}|_0}{n}\leq \frac{|\lambda^{k_n+1}|_{0}}{|\lambda^{k_n}|}\leq \frac{p|\lambda^{k_n+1}|+n(\beta)}{|\lambda^{k_n}|},
\end{equation}
and
\begin{equation}
\label{p lower}
\frac{p|\lambda^{k_n}|-n(\beta)}{|\lambda^{k_n+1}|}\leq \frac{|\lambda^{k_n}|_{0}}{|\lambda^{k_n+1}|}\leq \frac{|(\lambda_i)_{i=1}^{n}|_0}{n}.
\end{equation}
Importantly $|\lambda^{k_n+1}|-|\lambda^{k_n}|\leq n(\beta).$ Therefore as $n\to\infty$ the right hand side of \eqref{p upper} converges to $p$ and the left hand side of \eqref{p lower} converges to $p$. Therefore
$$\lim_{n\to\infty}\frac{|(\lambda_i)_{i=1}^{n}|_0}{n}=p$$ as required.

\end{proof}
Using Proposition \ref{freq prop} we obtain Theorem \ref{hybrid theorem} almost immediately. For completion we include a proof of this theorem.

\begin{proof}[Proof of Theorem \ref{hybrid theorem}]
Let $\beta\in(1,\frac{1+\sqrt{5}}{2})$ and $x\in [-2c,2c].$ Where $c$ is as in Proposition \ref{freq prop}. By Proposition \ref{freq prop} there exists $(\epsilon_i)\in\{-1,1\}^{\mathbb{N}}$ such that $x=\sum_{i=1}^{\infty}\epsilon_i\beta^{-i},$ $$\lim_{n\to \infty} \frac{\#\{1\leq i \leq n:\epsilon_i=1\}}{n}=\frac{1+x}{2}\textrm{ and }\lim_{n\to \infty} \frac{\#\{1\leq i \leq n:\epsilon_i=-1\}}{n}=\frac{1-x}{2}.$$ Therefore
\begin{align*}
\lim_{n\to\infty}\frac{1}{n}\sum_{i=1}^{n}\epsilon_i&=\lim_{n\to\infty}\Big(\frac{\#\{1\leq i \leq n:\epsilon_i=1\}}{n}- \frac{\#\{1\leq i \leq n:\epsilon_i=-1\}}{n}\Big)\\
&=\frac{1+x}{2}-\frac{1-x}{2}\\
&=x.
\end{align*}Consequently $(\epsilon_i)$ is a hybrid expansion of $x$.
\end{proof}

The following proposition implies statement $2$ from Theorem \ref{frequency theorem}. It is in fact a slightly stronger result. Before stating this proposition we introduce some notation. Given $(\epsilon_i)\in\{0,1\}^{\mathbb{N}}$ we define $$L(\epsilon_i):=\Big\{p\in[0,1]: p \textrm{ is an accumulation point of }\, \frac{\#\{1\leq i \leq n:\epsilon_i=0\}}{n}\Big\}.$$

\begin{proposition}
\label{no frequency prop}
Let $\beta\in(1,\frac{1+\sqrt{5}}{2}).$ There exists $c=c(\beta)>0$ such that for any $x\in(0,\frac{1}{\beta-1}),$ there exists a sequence $(\epsilon_i)\in\{0,1\}^{\mathbb{N}}$ such that $(\epsilon_i)$ is $\beta$-expansion of $x$ and $$\Big[\frac{1}{2}-c,\frac{1}{2}+c\Big]\subseteq L(\epsilon_i).$$
\end{proposition}

\begin{proof}
Let $\beta\in(1,\frac{1+\sqrt{5}}{2})$ and $x\in(0,\frac{1}{\beta-1})$. Just as in Proposition \ref{freq prop} there is no loss of generality in assuming that $x\in\mathcal{O}_{\beta}.$ We also let $J:=[\frac{1}{2}-\frac{1}{2n(\beta)},\frac{1}{2}+\frac{1}{2n(\beta)}]$ be as in Proposition \ref{freq prop}. Let $D\subseteq J$ be a countable dense subset of $J$ consisting of elements of the interior of $J$. This interior condition will be useful in our proof. Now let $(y_k)$ be a sequence consisting of elements of $D$ such that each element of $D$ appears infinitely often. We will show that there exists $(\lambda_i)\in\Omega_{\beta}(x),$ such that for each $k\in \mathbb{N}$ there exists $n_{k}\in\mathbb{N}$ for which we have
\begin{equation}
\label{subsequence equation}
\big||(\lambda_i)_{i=1}^{n_{k}}|_{0}-y_kn_{k}\big|\leq n(\b).
\end{equation}Here $n(\beta)$ is as in Proposition \ref{important prop}. The sequence $(n_{k})$ we construct will be strictly increasing. Since each element of $D$ appears infinitely often in $(y_k),$ and $D$ is dense in $J,$ it will follow from \eqref{subsequence equation} and Lemma \ref{Bijection lemma} that there exists $(\epsilon_i)\in \Sigma_{\b}(x)$ such that $J \subseteq L(\epsilon_i).$
\\

\noindent \textbf{Step $1$. } Suppose $y_1\in[1/2,1/2+(2n(\beta))^{-1}],$ then we apply $\omega^0$ to $x$. Here $\omega^0$ is the map guaranteed by Proposition \ref{important prop}. We let $\lambda^1=\omega^0$ and observe that
\begin{equation*}
0\leq \Big(\frac{1}{2}+\frac{1}{2n(\beta)}\Big) |\lambda^1|- y_1|\lambda| \leq |\lambda^1|_0-y_1|\lambda^1|\leq |\lambda^1|_0\leq n(\beta).
\end{equation*}
The second inequality follows from \eqref{positive growth}. The last inequality is a consequence of Proposition \ref{important prop}. Clearly we have
\begin{equation}
\label{step1 subsequence}
 \big||\lambda^1|_0-y_1|\lambda^1|\big|\leq n(\beta)
\end{equation} and $\lambda^1(x)\in\mathcal{O}_{\beta}$ by Proposition \ref{important prop}. Similarly, if $y_1\in [1/2-(2n(\beta))^{-1},1/2],$ then we let $\lambda^1=\omega^1$ and obtain \eqref{step1 subsequence} and $\lambda^1(x)\in\mathcal{O}_{\beta}$.
\\

\noindent \textbf{Step $k+1$. } Assume we have constructed $\lambda^{k}\in \{T_{0},T_1\}^{*}$  and $n_1,\ldots, n_{k}\in\mathbb{N}$ such that
\begin{equation}
\label{close equation}
\big||(\lambda_i^k)_{i=1}^{n_j}|_0-y_jn_j\big|\leq n(\beta),
\end{equation} for all $1\leq j \leq k$ and $\lambda^{k}(x)\in\mathcal{O}_{\beta}$. We now construct $\lambda^{k+1}$ and $n_{k+1}$ so that $\lambda^{k+1}$ satisfies \eqref{close equation} for all $1\leq j \leq k+1$ and $\lambda^{k+1}(x)\in\mathcal{O}_{\beta}$.

Without loss of generality we may assume that $n_{k}=|\lambda^k|.$ Consider the quantity $$|\lambda^k|_0- y_{k+1}n_{k}.$$ This term is either positive or negative. Let us assume it is positive. In which case let $\lambda^{(k,1)}=(\lambda^k,\omega^1),$ where $\omega^1$ is the $\omega^1$ guaranteed by Proposition \ref{important prop}. Then $\lambda^{(k,1)}(x)\in\mathcal{O}_{\beta}$ and
\begin{align*}
|\lambda^{(k,1)}|_0-y_{k+1}|\lambda^{(k,1)}|&=|\lambda^{k}|_0+|\omega^1|_0-y_{k+1}(|\lambda^k|+|\omega^1|)\\
& =|\lambda^{k}|_0-y_{k+1}|\lambda^{k}|+|\omega^1|_0-y_{k+1}|\omega^1|\\
&\leq |\lambda^{k}|_0-y_{k+1}|\lambda^{k}|+ \Big(\frac{1}{2}-\frac{1}{2n(\beta)}-y_{k+1}\Big)|\omega^1|&&(\textrm{By } \eqref{negative growth})\\
&\leq |\lambda^{k}|_0-y_{k+1}|\lambda^{k}|+ \Big(\frac{1}{2}-\frac{1}{2n(\beta)}-y_{k+1}\Big).
\end{align*}
Combining the first and the last line we see that
$$|\lambda^{(k,1)}|_0-y_{k+1}|\lambda^{(k,1)}|\leq  |\lambda^{k}|-y_{k+1}|\lambda^{k}|+ \Big(\frac{1}{2}-\frac{1}{2n(\beta)}-y_{k+1}\Big).$$ As this point we ask whether $$|\lambda^{(k,1)}|_0-y_{k+1}|\lambda^{(k,1)}|$$ is positive or negative. If it is negative then
\begin{align*}
0\geq |\lambda^{(k,1)}|_0-y_{k+1}|\lambda^{(k,1)}|&=|\lambda^{k}|_0-y_{k+1}|\lambda^{k}|+|\omega^1|_0-y_{k+1}|\omega^1|\\
&\geq |\omega^1|_0-y_{k+1}|\omega^1|&&(\textrm{Since }|\lambda^k|_0\geq y_{k+1}n_{k})\\
&\geq -y_{k+1}|\omega^1|\\
&\geq -n(\beta)&&(\textrm{Since } |\omega^1|\leq n(\beta)).
\end{align*}In which case we satisfy $$\big||\lambda^{(k,1)}|_0-y_{k+1}|\lambda^{(k,1)}|\big|\leq n(\beta)$$ and $\lambda^{(k,1)}(x)\in\mathcal{O}_{\beta}.$ At this point we stop and let $\lambda^{k+1}=\lambda^{(k,1)}$. If $|\lambda^{(k,1)}|_0-y_{k+1}|\lambda^{(k,1)}|$ is positive then we apply $\omega^{1}$ to $\lambda^{(k,1)}(x)$ and let $\lambda^{(k,2)}=(\lambda^{(k,1)},\omega^1).$ Then $\lambda^{(k,2)}(x)\in\mathcal{O}_{\beta}$ and by the same arguments used above we can show that
$$|\lambda^{(k,2)}|_0-y_{k+1}|\lambda^{(k,2)}|\leq  |\lambda^{(k,1)}|_0-y_{k+1}|\lambda^{(k,1)}|+ \Big(\frac{1}{2}-\frac{1}{2n(\beta)}-y_{k+1}\Big).$$ If $|\lambda^{(k,2)}|_0-y_{k+1}|\lambda^{(k,2)}|$ is negative, then by repeating the arguments given above it can be shown that $$\big||\lambda^{(k,2)}|_0-y_{k+1}|\lambda^{(k,2)}|\big|\leq n(\beta)$$ and $\lambda^{(k,2)}(x)\in\mathcal{O}_{\beta}$. In which case we stop and take $\lambda^{k+1}=\lambda^{(k,2)}.$ If $|\lambda^{(k,2)}|_0-y_{k+1}|\lambda^{(k,2)}|$ is positive then we let $\lambda^{(k,3)}=(\lambda^{(k,2)},\omega^1)$ and consider $|\lambda^{(k,3)}|_0-y_{k+1}|\lambda^{(k,3)}|.$ If this term is negative then our algorithm terminates and we take $\lambda^{k+1}=\lambda^{(k,3)},$  if not we consider $\lambda^{(k,4)}$ and so on. Each time our algorithm repeats we obtain a sequence $\lambda^{(k,j+1)}$ such that
\begin{equation}
\label{drop equation2}
|\lambda^{(k,j+1)}|_0-y_{k+1}|\lambda^{(k,j+1)}|\leq  |\lambda^{(k,j)}|-y_{{k+1}}|\lambda^{(k,j)}|+ \Big(\frac{1}{2}-\frac{1}{2n(\beta)}-y_{k+1}\Big),
\end{equation} and $\lambda^{(k,j+1)}(x)\in\mathcal{O}_{\beta}$. Repeatedly applying \eqref{drop equation2} we obtain
\begin{equation}
\label{j+1 drop equation}
|\lambda^{(k,j+1)}|_0-y_{k+1}|\lambda^{(k,j+1)}|\leq  |\lambda^{k}|_0-y_{k+1}|\lambda^{k}|+ (j+1)\Big(\frac{1}{2}-\frac{1}{2n(\beta)}-y_{k+1}\Big).
\end{equation}Since $y_{k+1}$ is in the interior of $J$ it follows that
$\frac{1}{2}-\frac{1}{2n(\beta)}-y_{k+1}<0$. Consequently, there must exists $j\in\mathbb{N}$ such that $$|\lambda^{(k,j+1)}|_0-y_{k+1}|\lambda^{(k,j+1)}|\leq 0< |\lambda^{(k,j)}|_0-y_{k+1}|\lambda^{(k,j)}|.$$ At which point it can be shown that $$\big||\lambda^{(k,j+1)}|_0-y_{k+1}|\lambda^{(k,j+1)}|\big|\leq n(\beta)$$ and $\lambda^{(k,j+1)}(x)\in\mathcal{O}_{\beta}$. Taking $\lambda^{k+1}=\lambda^{(k,j+1)}$ and $n_{k+1}=|\lambda^{k+1}|,$ we see that we satisfy \eqref{close equation} for $1\leq i \leq k+1$ and $\lambda^{k+1}(x)\in\mathcal{O}_{\beta}$. The case where the initial quantity  $$|(\lambda_i)_{i=1}^{n_{k}}|_0- y_{k+1}n_{k}$$ is negative is handled similarly. In this case we keep applying $\omega^0$ until we see a sign change. Thus we have completed our inductive step.
\\

Repeatedly applying step $k+1$ yields an infinite limit sequence $\lambda\in\Omega_{\beta}(x)$. Since \eqref{close equation} holds for each $\lambda^k$ it follows that \eqref{subsequence equation} is satisfied by $\lambda$ and we have proved our result.

\end{proof}

Statements $2$ and $3$ from Theorem \ref{frequency theorem} follow from Proposition \ref{freq prop} and Proposition \ref{no frequency prop}. Statement $1$ of this theorem for the parameter space $(1,\frac{1+\sqrt{5}}{2})$ follows from Proposition \ref{freq prop}. Now we prove Statement $1$ from Theorem \ref{frequency theorem} for the parameter space $[\frac{1+\sqrt{5}}{2},\beta_{KL}).$

\begin{proof}[Proof of statement $1$ from Theorem \ref{frequency theorem} within the parameter space $[\frac{1+\sqrt{5}}{2},\beta_{KL})$.] Let us start by fixing $\beta\in[\frac{1+\sqrt{5}}{2},\beta_{KL}).$ Then there exists $n\in\mathbb{N}$ such that $\beta\in[\beta_n,\beta_{n+1}).$ Recall that the sequences $(\b_n)$ is defined in Section \ref{Section3}.

For each $1\leq i \leq n+1$ let
$$\mathcal{I}_i:=[\pi_{\beta}((\tau^{i})^{\infty}),\pi_{\beta}((\overline{\tau^i})^{\infty})].$$ Where the $\tau^i$ are the finite sequences appearing in the construction of the Thue-Morse sequence in Section \ref{Section3}.
By \eqref{switch1} and \eqref{switch2} we know that these intervals are well defined and
\begin{equation}
\label{nested}
\mathcal{I}_{n+1}\subseteq S_{\beta}\subseteq \mathcal{I}_{n}\subseteq \cdots \subseteq \mathcal{I}_{1}=\mathcal{O}_{\beta}.
\end{equation} Moreover, by \eqref{switch1} and \eqref{switch2} we know that $\mathcal{I}_{n+1}$ is a proper subinterval of $S_{\beta}.$ Therefore $$[T_1(\pi_{\beta}((\tau^{n+1})^{\infty}), T_0(\pi_{\beta}((\overline{\tau^{n+1}})^{\infty}))]\subseteq \Big(0,\frac{1}{\beta-1}\Big).$$ It follows from this observation, Lemma \ref{basic lemma}, and the expansivity of the maps $T_0$ and $T_1$ that if $x\in [\pi_{\beta}((\tau^{i})^{\infty}, \pi_{\beta}((\overline{\tau^i})^{\infty})],$ then $T_{0}(x)$ and $T_{1}(x)$ can both be mapped back into $\mathcal{O}_{\beta}$ using at most $l(\beta)\in\mathbb{N}$ iterations of $T_1$ or $T_0$ respectively. Importantly $l(\beta)$ is a natural number that only depend upon $\beta$.

Now let us fix $x\in(0,\frac{1}{\beta-1}).$ Without loss of generality we may assume $x\in\mathcal{O}_{\beta}$. If $x$ is a preimage of an endpoint of an $\mathcal{I}_{i},$ then by Lemma \ref{tau normal} we know that $x$ has a simply normal expansion. Therefore to prove our result it suffices to consider those $x$ that are not preimages of an endpoint of an $\mathcal{I}_{i}.$ We now give an algorithm which shows how one can construct a simply normal expansion for any $x$ satisfying this condition.
\\

\noindent \textbf{Step $1$.}  By \eqref{nested} and our assumption that $x$ is not a preimage of an endpoint of an $\mathcal{I}_{i},$ we know that $x$ satisfies one of the following:
$$x\in(\pi_{\beta}((\tau^{n+1})^{\infty}),\pi_{\beta}((\overline{\tau^{n+1}})^{\infty})),\, x\in (\pi_{\beta}((\tau^{i})^{\infty}),\pi_{\beta}((\tau^{i+1})^{\infty}))$$ or $$x\in(\pi_{\beta}((\overline{\tau^{i+1}})^{\infty}),\pi_{\beta}((\overline{\tau^{i}})^{\infty}))$$for some $1\leq i \leq n$. If $x\in (\pi_{\beta}((\tau^{n+1})^{\infty}),\pi_{\beta}((\overline{\tau^{n+1}})^{\infty}))$ apply $T_0$ to $x$ and then $T_1$ until $(T_1^j\circ T_0)(x)\in\mathcal{O}_{\beta}.$ By our previous remarks we know that $j\leq l(\beta)$. Let $\lambda^1=(T_0,(T_1)^j)$. Then  $$\big||(\lambda^1)_{i=1}^m|_0-|(\lambda^1)_{i=1}^m|_{1}|\big|\leq l(\beta)$$ for all $1\leq m \leq |\lambda^1|.$

If $x\in (\pi_{\beta}((\tau^{i})^{\infty}),\pi_{\beta}((\tau^{i+1})^{\infty})),$ then we repeatedly apply $\kappa^i$ to $x$ until $x$ is mapped into $(\pi_{\beta}((\tau^{i+1})^{\infty}), \pi_{\beta}((\overline{\tau^{i+1}})^{\infty})).$ This follows from \eqref{flip equation}, our assumption that $x$ is not a preimage of an endpoint of an $\mathcal{I}_i$, and the fact that $\pi_{\beta}((\tau^{i})^{\infty})$ is the unique fixed point of $\kappa^i$ and $\kappa^i$ scales distances by some factor strictly greater than one. Likewise, if $x\in(\pi_{\beta}((\overline{\tau^{i+1}})^{\infty}),\pi_{\beta}((\overline{\tau^{i}})^{\infty})),$ then by repeatedly applying $\overline{\kappa}^i$ the point $x$ is mapped into $(\pi_{\beta}((\tau^{i+1})^{\infty}), \pi_{\beta}((\overline{\tau^{i+1}})^{\infty})).$ In either case we let $x^1$ denote the image point of $x$ in $(\pi_{\beta}((\tau^{i+1})^{\infty}), \pi_{\beta}((\overline{\tau^{i+1}})^{\infty})).$ If $x^1\notin (\pi_{\beta}((\tau^{n+1})^{\infty}),\pi_{\beta}((\overline{\tau^{n+1}})^{\infty}))$ then
\begin{equation}
\label{dumbdumb1}
x^1\in(\pi_{\beta}((\tau^{i_1})^{\infty}),\pi_{\beta}((\tau^{i_1+1})^{\infty}))\cup(\pi_{\beta}((\overline{\tau^{i_1+1}})^{\infty}),\pi_{\beta}((\overline{\tau^{i_1}})^{\infty}))
\end{equation}for some $i_1>i.$

Repeating the above argument, we see that if $x^1\notin (\pi_{\beta}((\tau^{n+1})^{\infty}),\pi_{\beta}((\overline{\tau^{n+1}})^{\infty})),$ then by repeatedly apply either $\kappa^{i_1}$ or $\overline{\kappa}^{i_1}$ to $x^1$ our orbit is eventually mapped into $(\pi_{\beta}((\tau^{i_1+1})^{\infty}),\pi_{\beta}((\overline{\tau^{i_1+1}}^{\infty}))).$ We can repeat this procedure until our orbit is eventually mapped in to $(\pi_{\beta}((\tau^{n+1})^{\infty}),\pi_{\beta}((\overline{\tau^{n+1}})^{\infty})).$ Therefore we may conclude that there exists a sequence of maps $\kappa^{*}\in\{T_0,T_1\}^{*}$ such that $$\kappa^{*}(x)\in (\pi_{\beta}((\tau^{n+1})^{\infty}),\pi_{\beta}((\overline{\tau^{n+1}})^{\infty})).$$ Moreover, the sequence of maps $\kappa^*$ is the concatenation of finitely many blocks all of length at most $2^n$, where each of these blocks have the same number of $T_0$'s and $T_1$'s by Lemma \ref{tau normal}. Therefore $$|\kappa^{*}|_0=|\kappa^{*}|_1\textrm{ and }\big||(\kappa^{*})_{i=1}^m|_0-|(\kappa^{*})_{i=1}^m|_1\big|\leq 2^n$$ for all $1\leq m \leq |\kappa^{*}|$. We now apply $T_0$ to $\kappa^*(x)$ and then apply $T_1$ until $(T_1^j\circ T_0\circ\kappa^*)(x)\in\mathcal{O}_{\beta}.$ Let $\lambda^1=(\kappa^*,T_0,(T_1)^j)$. Then $\lambda^1(x)\in\mathcal{O}_{\beta}$ and $$\big||(\lambda^{1}_i)_{i=1}^m|_0-|(\lambda^{1}_i)_{i=1}^m|_1\big|\leq 2^n$$ if $1\leq m \leq |\kappa^{*}|.$ Moreover, $$\big||(\lambda^{1}_i)_{i=1}^m|_0-|(\lambda^{1}_i)_{i=1}^m|_1\big|\leq l(\beta)$$ if $|\kappa^{*}|< m\leq |\lambda^1|$ since $|\kappa^{*}|_0=|\kappa^{*}|_1$ and $j\leq l(\beta)$.

It follows from the above that we have constructed $\lambda^1\in\{T_0,T_1\}^*$ such that $\lambda^{1}(x)\in\mathcal{O}_{\beta},$
\begin{equation}
\label{one time}
\big||(\lambda^{1}_i)_{i=1}^m|_0-|(\lambda^{1})_{i=1}^m|_{1}|\big|\leq 2^n+l(\beta)
\end{equation} for all $1\leq m \leq |\lambda^1|,$ and
\begin{equation}
\label{two time}
\big||\lambda^1|_0-|\lambda^1|_1\big|\leq l(\beta).
\end{equation}
\\

\noindent \textbf{Step $k+1$. }Suppose we have constructed $\lambda^{k}\in \{T_0,T_1\}^{*}$ such that $\lambda^{k}(x)\in\mathcal{O}_{\beta},$
\begin{equation}
\label{n balanced}
\big|(\lambda^{k}_i)_{i=1}^m|_0-|(\lambda^{k})_{i=1}^m|_{1}\big|\leq 2^n+  l(\beta)
\end{equation} for all $1\leq m \leq |\lambda^k|,$ and
\begin{equation}
\label{balanced k}
\big||\lambda^{k}|_0-|\lambda^{k}|_1\big|\leq l(\beta).
\end{equation} We now show how to construct $\lambda^{k+1}$ satisfying $\lambda^{k+1}(x)\in\mathcal{O}_{\beta}$, \eqref{n balanced}, and \eqref{balanced k}. There are two cases to consider. Either $$|\lambda^{k}|_0-|\lambda^{k}|_1$$ is positive, or it is negative. Let us assume it is positive. The negative case is handled similarly. By the same argument used in Step $1,$ if $\lambda^{k}(x)\notin (\pi_{\beta}((\tau^{n+1})^{\infty}),\pi_{\beta}((\overline{\tau^{n+1}})^{\infty})),$ then there exists $\kappa^{*}\in\{T_{0},T_{1}\}^*$ such that $|\kappa^*|_0=|\kappa^*|_1$ and $(\kappa^*\circ\lambda^{k})(x)\in (\pi_{\beta}((\tau^{n+1})^{\infty}),\pi_{\beta}((\overline{\tau^{n+1}})^{\infty}))$. Moreover $\kappa^{*}$ is the concatenation of finitely many blocks of length at most $2^n,$ and each block has the same number of $T_{0}$'s and $T_{1}$'s. We then apply $T_0$ and $T_1$ until $(T_1^j\circ T_0\circ\kappa^*\circ\lambda^{k})(x)\in\mathcal{O}_{\beta}.$ At this point we set $\lambda^{k+1}=(\lambda^{k},\kappa^{*},T_0,T_1^{j}).$ Then
$$\big||(\lambda^{k+1}_i)_{i=1}^{m}|_0-|(\lambda^{k+1}_i)_{i=1}^{m}|_1\big|\leq 2^n+l(\beta)$$ if $1\leq m\leq |\lambda^{k}|$ by \eqref{n balanced}. If $|\lambda^{k}|< m\leq |\lambda^{k}|+|\kappa^{*}|$ then $$\big||(\lambda^{k+1}_i)_{i=1}^{m}|_0-|(\lambda^{k+1}_i)_{i=1}^{m}|_1\big|\leq 2^n+l(\beta).$$ This is a consequence of \eqref{balanced k} and the fact that $\kappa^*$ is the concatenation of finitely many blocks of length at most $2^n,$ where each block has the same number of $T_0$'s as $T_1$'s. If $ |\lambda^{k}|+|\kappa^{*}|< m\leq |\lambda^{k+1}|$ then
\begin{equation}
\label{big lad}
|(\lambda^{k+1}_i)_{i=1}^{m}|_0-|(\lambda^{k+1})_{i=1}^{m}|_{1}= |\lambda^{k}|_0-|\lambda^{k}|_{1}+ |\kappa^{*}|_0-|\kappa^{*}|_1+1
-(m-|\lambda^{k}|-|\kappa^{*}|-1).
\end{equation}Using the fact that $|\kappa^{*}|_0=|\kappa^{*}|_1$ and \eqref{balanced k}, we see that \eqref{big lad} implies $$|(\lambda^{k}_i)_{i=1}^{m}|_0-|(\lambda^{k}_i)_{i=1}^{m}|_{1}\leq l(\beta)+1$$ if $|\lambda^{k}|+|\kappa^{*}|< m\leq  |\lambda^{k+1}|$. Using the assumption $|\lambda^{k}|_0-|\lambda^{k}|_{1}$ is positive, along with $|\kappa^{*}|_0=|\kappa^{*}|_1$ and $ j\leq l(\beta),$ we see that \eqref{big lad} also implies
$$-l(\beta)\leq |(\lambda^{k}_i)_{i=1}^{m}|_0-|(\lambda^{k})_{i=1}^{m}|_{1}$$ if $|\lambda^{k}|+|\kappa^{*}|< m\leq |\lambda^{k+1}|.$ Therefore $$\big| |(\lambda^{k}_i)_{i=1}^{m}|_0-|(\lambda^{k})_{i=1}^{m}|_{1}\big|\leq 2^n+l(\beta)$$ if $|\lambda^{k}|+|\kappa^{*}|< m\leq |\lambda^{k+1}|.$ Moreover, since $j\geq 1$ we see that \eqref{big lad} implies $$|\lambda^{k+1}|_0-|\lambda^{k+1}|_1\leq |\lambda^{k}|_0-|\lambda^{k}|_1\leq l(\beta).$$ Therefore $\lambda^{k+1}(x)\in\mathcal{O}_{\beta}$ and $\lambda^{k+1}$ satisfies \eqref{n balanced} and \eqref{balanced k}. We have completed our inductive step when $\lambda^{k}(x)\notin (\pi_{\beta}((\tau^{n+1})^{\infty}),\pi_{\beta}((\overline{\tau^{n+1}})^{\infty}))$. When $\lambda^{k}(x)\in (\pi_{\beta}((\tau^{n+1})^{\infty}),\pi_{\beta}((\overline{\tau^{n+1}})^{\infty}))$ the construction of $\lambda^{k+1}$ is the same as above except we do not need to construct the sequence of maps $\kappa^{*}.$
\\

Repeating this procedure indefinitely gives rise to an infinite sequence $\lambda\in\Omega_{\beta}(x)$ such that
\begin{equation}
\label{final limit}
\big||(\lambda_i)_{i=1}^m|_0- |(\lambda_i)_{i=1}^m|_1\big|\leq 2^n+l(\beta)
\end{equation} for all $m\in\mathbb{N}$. It follows from \eqref{final limit} that within $\lambda$ the map $T_0$ appears with frequency $1/2$ and the map $T_1$ appears with frequency $1/2$. By Lemma \ref{Bijection lemma} there exists $(\epsilon_i)\in\Sigma_{\beta}(x)$ that is simply normal.

\end{proof}

\subsection{Proofs for Theorem \ref{exceptional theorem} and Theorem \ref{slow growth theorem}}
We now give a proof of Theorem \ref{exceptional theorem}.

\begin{proof}[Proof of Theorem \ref{exceptional theorem}]Recall from Lemma \ref{lexicographic lemma} that
\begin{equation}
\label{multi lex}
\widetilde{\mathcal{A}}_{\beta}=\Big\{(\epsilon_i)\in\{0,1\}^{\mathbb{N}}:(\overline{\alpha_i(q)})\prec(\epsilon_{n+i})\prec(\alpha_i(q))\textrm{ for all }n\in\mathbb{N}\Big\}.
\end{equation}Let $\beta_{n}$ be the unique positive solution to the equation $$x^{n+1}=x^{n}+x^{n-1}+\cdots+x+1$$ with modulus larger than $1$. The number $\beta_n$ is commonly referred to as the $n$-th multinacci number. Note that $\beta_n\nearrow 2$ as $n\to\infty$. It is a consequence of Lemma \ref{quasi greedy properties} that $$\alpha(\beta_n)=((1)^n,0)^{\infty}.$$ It follows from \eqref{multi lex} that
\begin{equation}
\label{dumbinclusion}
\Big\{(\epsilon_i)\in\{0,1\}^{\mathbb{N}}:(\epsilon_i) \textrm{ does not contain } n \textrm{ consecutive }0's \textrm{ or }1's\Big\}\subseteq \widetilde{\mathcal{A}}_{\beta_n}.
\end{equation}
Let $$W_n=\Big\{(\epsilon_i)\in\{0,1\}^n:|(\epsilon_i)|_1>|(\epsilon_i)|_0 \textrm{ and }(\epsilon_i)\neq (1)^{n}\Big\}.$$
Consider the case where $n=2k+1$. Any element of $\{0,1\}^{2k+1}$ satisfies either $|(\epsilon_i)|_1>|(\epsilon_i)|_0$ or $|(\epsilon_i)|_1>|(\epsilon_i)|_0$. It follows that
\begin{equation}
\label{count equation}
\#W_{2k+1}=2^{2k}-1.
\end{equation}
Let $T_{2k+1}:=W_{2k+1}^{\mathbb{N}}.$ Each element of $T_{2k+1}$ fails to be simply normal. This is because the number of $1$'s in each successive block of length $2k+1$ is at least $k$. What is more, any element of $T_{2k+1}$ cannot contain $2(2k+1)$ consecutive $0$'s or $1$'s. Therefore $T_{2k+1}\subseteq \widetilde{\mathcal{A}}_{\beta_{2(2k+1)}}$ by \eqref{dumbinclusion}. By Lemma \ref{inclusion lemma} we also know that $T_{2k+1}\subseteq \widetilde{A}_{\beta}$ for any $\beta\in(\beta_{2(2k+1)},2)$.

We now compute the Hausdorff dimension of the set $\pi_{\beta}(T_{2k+1})$ for $\beta\in(\beta_{2(2k+1)},2)$. Since every element of $T_{2k+1}$ fails to be simply normal and each element of $\pi_{\beta}(T_{2k+1})$ has a unique $\beta$-expansion, the Hausdorff dimension of $\pi_{\beta}(T_{2k+1})$ will give a lower bound for the Hausdorff dimension of those $x$ without a simply normal $\beta$-expansion.

Let us now fix $\beta\in(\beta_{2(2k+1)},2)$. Notice that $\pi_{\beta}(T_{2k+1})$ satisfies the similarity relation
\begin{equation}
\label{self similar}
\pi_{\beta}(T_{2k+1})=\bigcup_{(\epsilon_i)_{i=1}^{2k+1}\in W_{2k+1}^1}(T^{-1}_{\epsilon_1}\circ \cdots \circ T^{-1}_{\epsilon_{2k+1}})(\pi_{\beta}(T_{2k+1})).
\end{equation}Each map on the right hand side of \eqref{self similar} is a contracting similarity that scales by a factor $\beta^{-2k-1}$. Therefore $\pi_{\beta}(T_{2k+1})$ is a self-similar set. It is a consequence of each element of $\pi_{\beta}(T_{2k+1})$ having a unique $\beta$-expansion that the union in \eqref{self similar} is disjoint. Therefore $\pi_{\beta}(T_{2k+1})$ is a self-similar set and the IFS generating it satisfies the strong separation condition. The well known formula for the Hausdorff dimension of a self-similar set satisfying the strong separation condition, see for example \cite{Fal}, implies that $\dim_{H}(\pi_{\beta}(T_{2k+1}))$ satisfies $$1=\#W_{2k+1}\cdot\beta^{-(2k+1)\dim_{H}(\pi_{\beta}(T_{2k+1}))}.$$ Rearranging this equation and appealing to \eqref{count equation} we obtain
$$\dim_{H}( \pi_{\beta}(T_{2k+1}))= \frac{\log 2^{2k} -1}{\log \beta^{2k+1}}>\frac{\log 2^{2k} -1}{\log 2^{2k+1}}\geq \frac{2k-1}{2k+1}$$ for any $\beta\in(\beta_{2(2k+1)},2).$ Since $k$ is arbitrary it follows that $$\lim_{\beta\nearrow 2}\dim_{H}\Big(\Big\{x: x \textrm{ has no simply normal } \beta\textrm{-expansion}\Big\}\Big)=1.$$

\end{proof}

We now give a proof of Theorem \ref{slow growth theorem}. In the proof of this theorem we will require the interpretation of Proposition \ref{important prop} when the digit set is $\{-1,1\}$ not $\{0,1\}$.

\begin{proof}[Proof of Theorem \ref{slow growth theorem}]
Let us start by fixing $\beta\in(1,\frac{1+\sqrt{5}}{2})$ and $x\in(\frac{-1}{\beta-1},\frac{1}{\beta-1}).$ Suppose $f:\mathbb{N}\to\mathbb{R}$ is a strictly increasing function satisfying $$\lim_{n\to\infty} f(n)=\infty$$ and
\begin{equation}
\label{explicit decay}
f(n+1)-f(n)<\frac{\beta-1}{n(\beta)}
\end{equation} for all $n\geq N,$ where $N$ is some large natural number. Here $n(\beta)$ is as in the statement of Proposition \ref{important prop}. We now describe an algorithm which yields an expansion of $x$ with the desired properties.
\\

\noindent \textbf{Step $1$.} The first step in our construction is to pick an arbitrary sequence $\lambda^0\in\{T_{-1},T_1\}^{N}$ such that $\lambda^0(x)\in \mathcal{O}_{\beta}$. We can do this by Lemma \ref{basic lemma} and replacing our value of $N$ with a larger value if necessary. At this point we consider the sign of the quantity
\begin{equation}
\label{sign quantity}
|\lambda^0|_1-|\lambda^0|_{-1} - f(N)x.
\end{equation}Let us start by assuming this quantity is negative. Since $\lambda^0(x)\in \mathcal{O}_{\beta},$ we can apply Proposition \ref{important prop} to assert that there exists $\omega^{1}$ satisfying $(\omega^{1}\circ\lambda^0)(x)\in\mathcal{O}_{\beta},$ $|\omega^{1}|\leq n(\beta)$, and $|\omega^{1}|_{-1}<|\omega^{1}|_{1}.$ Let $\lambda^{0,1}=(\lambda^0,\omega^{1})$. Consider the quantity
$$|\lambda^{0,1}|_1-|\lambda^{0,1}|_{-1} - f(|\lambda^{0,1}|)x.$$ If this term is greater than or equal to zero then there has been a sign change. In which case let $\lambda^1=\lambda^{0,1}$ and observe
\begin{align*}
0&\leq  |\lambda^{1}|_1-|\lambda^{1}|_{-1} - f(|\lambda^{1}|)x\\
&=(|\lambda^{0}|_1-|\lambda^{0}|_{-1})+(|\omega^{1}|_1-|\omega^{1}|_{-1})-x\Big(f(|\lambda^0|)+\sum_{i=0}^{|\lambda^{1}|-|\lambda^0|-1}f(|\lambda^{1}|-i)-f(|\lambda^{1}|-i-1)\Big)\\
&=(|\lambda^{0}|_1-|\lambda^{0}|_{-1}-xf(|\lambda^0|))+(|\omega^{1}|_1-|\omega^{1}|_{-1})-x\Big(\sum_{i=0}^{|\lambda^{1}|-|\lambda^0|-1}f(|\lambda^{1}|-i)-f(|\lambda^{1}|-i-1)\Big)\\
&\leq 0 +n(\beta) + \Big|\frac{xn(\beta)(\beta-1)}{n(\beta)}\Big|\\
&\leq n(\beta)+1. &&
\end{align*}In the penultimate line we have used \eqref{explicit decay} and the fact that $|\omega^{1}|\leq n(\beta)$. Summarising the above, we have shown that
\begin{equation}
\label{just below equation}
0\leq |\lambda^{1}|_1-|\lambda^{1}|_{-1} - f(|\lambda^{1}|)x\leq n(\beta)+1
 \end{equation}if there has been a sign change. Suppose we do not see a sign change. By Proposition \ref{important prop} there exists $\omega^{1}$ satisfying $|\omega^1|\leq n(\beta),$ $|\omega^{1}|_{-1}<|\omega^{1}|_{1}$ and $\lambda^{0,2}(x)\in\mathcal{O}_{\beta},$ where $\lambda^{0,2}=(\lambda^{0,1},\omega^{1})$. We consider the quantity $$|\lambda^{0,2}|_1-|\lambda^{0,2}|_{-1} - f(|\lambda^{0,2}|)x,$$ and ask whether there has been a sign change. If there has been a sign change we let $\lambda^1=\lambda^{0,2}$. If not we concatenate $\lambda^{0,2}$ with the $\omega^{1}$ guaranteed by Proposition \ref{important prop}. We repeat this procedure and obtain a sequence $(\lambda^{0,j}).$ Note that for all $j\geq 1$ we have
\begin{equation}
\label{boring}
(|\lambda^{0,j+1}|_1-|\lambda^{0,j+1}|_{-1})-(|\lambda^{0,j}|_1-|\lambda^{0,j}|_{-1})\geq 1.
\end{equation} What is more,
\begin{align}
|xf(|\lambda^{0,j+1}|)-xf(|\lambda^{0,j}|)|&=\Big|x\Big(\sum_{i=0}^{|\lambda^{0,j+1}|-|\lambda^{0,j}|-1}f(|\lambda^{0,j+1}|-i)-f(|\lambda^{0,j+1}|-i-1)\Big)\Big|\nonumber\\
&<\big| \frac{x(\beta-1)n(\beta)}{n(\beta)}\big|\nonumber\\
&<c.\label{boring2}
\end{align} For some $c<1$ depending on $x$. Combining equations \eqref{boring} and \eqref{boring2} we obtain
\begin{equation}
\label{first drop}
|\lambda^{0,j+1}|_1-|\lambda^{0,j+1}|_{-1}-xf(|\lambda^{0,j+1}|)> |\lambda^{0,j}|_1-|\lambda^{0,j}|_{-1}-xf(|\lambda^{0,j}|)+(1-c).
\end{equation}
Repeatedly applying $\eqref{first drop}$ we observe that
\begin{equation}
\label{mth drop}
|\lambda^{0,j}|_1-|\lambda^{0,j}|_{-1}-xf(|\lambda^{0,j}|)> |\lambda^{0}|_1-|\lambda^{0}|_{-1}-xf(|\lambda^{0}|)+j(1-c).
\end{equation}Since $(1-c)>0$ equation \eqref{mth drop} implies that we must observe a sign change after finitely many steps. Let $\lambda^1=\lambda^{0,j^*}$ where $j^*$ is the smallest $j^*\in\mathbb{N}$ such that $$|\lambda^{0,j^*}|_1-|\lambda^{0,j^*}|_{-1} - f(|\lambda^{0,j^*}|)x\geq 0.$$ Repeating the calculation done above in the derivation of \eqref{just below equation}, it can be shown that $\lambda^1$ satisfies $$0\leq |\lambda^{1}|_1-|\lambda^{1}|_{-1} - f(|\lambda^1|)x\leq n(\beta)+1.$$ Moreover $\lambda^{1}(x)\in\mathcal{O}_{\beta}.$

The case where \eqref{sign quantity} is positive is dealt with slightly differently. This time we concatenate with $\omega^{-1}$'s until we observe a sign change. The sign change is guaranteed because the $|\lambda^{0,j+1}|_1-|\lambda^{0,j+1}|_{-1}$ term will be decreasing and the $f(|\lambda^{0,j}|)x$ term will be varying monotonically at a slower rate. By a simple calculation, when we observe a sign change we will have constructed a sequence $\lambda^1\in\{T_{-1},T_1\}^{*}$ such that $\lambda^{1}(x)\in\mathcal{O}_{\beta}$ and
\begin{equation}
\label{just above equation}
-n(\beta)-1\leq |\lambda^{1}|_1-|\lambda^{1}|_{-1} - f(|\lambda^1|)x \leq 0.
\end{equation}

Combining \eqref{just below equation} and \eqref{just above equation}, we see that in either case we have constructed $\lambda^1\in\{T_{-1},T_{1}\}^{*}$ such that $\lambda^1(x)\in\mathcal{O}_{\beta}$ and
\begin{equation}
\label{C equation}
\big||\lambda^{1}|_1-|\lambda^{1}|_{-1}-f(|\lambda^1|)x\big| \leq n(\beta)+1.
\end{equation}
\\

\noindent \textbf{Step $k+1$. }Suppose we have constructed $\lambda^{k}\in\{T_{-1},T_{1}\}^{*}$ such that $\lambda^{k}(x)\in\mathcal{O}_{\beta}$ and
 \begin{equation}
\label{stepk equation}
\big||\lambda^{k}|_1-|\lambda^{k}|_{-1}-f(|\lambda^{k}|)x\big| \leq n(\beta)+1.
\end{equation}
We now show how to construct $\lambda^{k+1}$ such that $\lambda^{k+1}(x)\in\mathcal{O}_{\beta}$ and \eqref{stepk equation} is still satisfied. Consider the term appearing within the modulus signs in \eqref{stepk equation}, if this term is positive then we concatenate $\lambda^{k}$ with $\omega^{-1},$ if it is negative then we concatenate $\lambda^{k}$ with $\omega^1$. Here $\omega^{-1}$ and $\omega^1$ are as in Proposition \ref{important prop}. In either case we call our new sequence $\lambda^{k+1}$. By Proposition \ref{important prop} we have $\lambda^{k+1}(x)\in\mathcal{O}_{\beta}.$ Moreover repeating the arguments given above one can show that
\begin{equation*}
\big||\lambda^{k+1}|_1-|\lambda^{k+1}|_{-1}-f(|\lambda^{k+1}|)x \big|\leq n(\beta)+1.
\end{equation*}This completes our inductive step.
\\

Note that it is a consequence of our construction that
\begin{equation}
\label{balancer}
|\lambda^{k+1}|-|\lambda^k|\leq n(\beta)
\end{equation} for all $k\geq 1$. Now let $\lambda\in \Omega_{\beta}(x)$ denote the infinite sequence of transformations we obtain by repeating step $k+1$ indefinitely. It is a consequence of \eqref{explicit decay}, \eqref{stepk equation} and \eqref{balancer} that
\begin{equation}
\label{balancer1}
\big||(\lambda_i)_{i=1}^n|_1-|(\lambda_i)_{i=1}^n|-f(n)x\big|\leq C(\beta)
\end{equation}for all $n\geq |\lambda^1|.$ Where $C(\beta)$ is a constant that only depends upon $\beta$.

Let $(\epsilon_i)$ be the element of $\Sigma_{\beta}(x)$ obtained by applying the bijection in Lemma \ref{Bijection lemma} to $\lambda$. Then using the simple identity $$\sum_{i=1}^n\epsilon_i=|(\lambda_i)_{i=1}^n|_1-|(\lambda_i)_{i=1}^n|$$ and \eqref{balancer1} we obtain
\begin{equation}
\label{C'' equation}
\big| \sum_{i=1}^{n}\epsilon_i-f(n)x\big| \leq C(\beta)
\end{equation}for all $n\geq |\lambda^1|$. Since $f(n)\to \infty$ we must have $$\lim_{n\to\infty}\frac{1}{f(n)}\sum_{i=1}^{n}\epsilon_i\to x$$ as required.
\end{proof}

\section{Self-affine sets with non-empty interior}
\label{affine section}
In this section we prove Theorem \ref{Affine theorem}. As we will see in Section \ref{Section Explicit calculations}, one can explicitly calculate a lower bound for the value of $\delta$ appearing in the statement of this theorem. We start by introducing some notation and proving a technical proposition.

Note that if $x\in\widetilde{\mathcal{O}}_{\beta}$ then $\omega^1(x)\in\widetilde{\mathcal{O}}_{\beta}$ and $\omega^{-1}(x)\in\widetilde{\mathcal{O}}_{\beta}$. Where $\omega^1$ and $\omega^{-1}$ are as in Proposition \ref{important prop}. Applying Proposition \ref{important prop} again, we know that there exists $\omega^{1'}$ and $\omega^{-1'}(x)$ such that $(\omega^{1'}\circ \omega^{1})(x)\in\widetilde{\mathcal{O}}_{\beta}$ and $(\omega^{-1'}\circ \omega^{-1})(x)\in\widetilde{\mathcal{O}}_{\beta}.$ Clearly we can apply Proposition \ref{important prop} repeatedly to $x$ and its successive images. By an abuse of notation, we let $(\omega^{1}_i)_{i=1}^{\infty}\in\widetilde{\Omega}_{\beta}(x)$ denote the infinite sequence we obtain by repeatedly applying $\omega^1$. Similarly $(\omega^{-1}_i)_{i=1}^{\infty}\in\widetilde{\Omega}_{\beta}(x)$ will denote the infinite sequence we obtain by repeatedly applying $\omega^{-1}$. Moreover, given an $x\in \widetilde{\mathcal{O}}_{\beta},$ a sequence whose entries consist of $\omega^{-1}$'s and $\omega^1$'s will represent the element of $\widetilde{\Omega}_{\beta}(x)$ obtained by repeatedly applying Proposition \ref{important prop} and applying $\omega^{-1}$ and $\omega^1$ in accordance with the order they appear in that sequence. In what follows we let $B:\{T_-1,T_1\}^{\mathbb{N}}\to \{-1,1\}^{\mathbb{N}}$ be the map which sends $(T_{\epsilon_i})$ to $(\epsilon_i)$. Note that $B$ is a bijection between $\widetilde{\Omega}_{\beta}(x)$ and $\widetilde{\Sigma}_{\beta}(x)$ by Lemma \ref{Bijection lemma}. By an abuse of notation we also let $B$ denote the map $B:\{T_-1,T_1\}^{n}\to \{-1,1\}^{n}$ which sends $(T_{\epsilon_i})_{i=1}^n$ to $(\epsilon_i)_{i=1}^n.$

Returning to our self-affine sets one can verify that $\Lambda_{\b_1,\b_2,\b_3}$ has the following closed from
$$\Lambda_{\b_1,\b_2,\b_3}=\Big\{\Big(\sum_{i=1}^{\infty}\frac{\epsilon_i}{\beta_1^i},\sum_{i=1}^{\infty}\frac{\epsilon_i}{\beta_2^{|(\epsilon_1)_{j=1}^i|_{-1}}\beta_3^{|(\epsilon_1)_{j=1}^i|_{1}}}\Big): (\epsilon_i)\in\{-1,1\}^\mathbb{N}\Big\}.$$ In what follows we let $\pi_{\beta_2,\beta_3}:\{-1,1\}^{\mathbb{N}}\to \mathbb{R}$ denote the map $$\pi_{\b_1,\b_2}((\epsilon_i))=\sum_{i=1}^{\infty}\frac{\epsilon_i}{\beta_2^{|(\epsilon_1)_{j=1}^i|_{-1}}\beta_3^{|(\epsilon_1)_{j=1}^i|_{1}}}.$$ The following equality holds for any $x\in [-\frac{1}{\beta_1-1},\frac{1}{\beta_1-1}]$
\begin{equation}
\label{fibre identification}
\pi_{\beta_2,\beta_3}(\widetilde{\Sigma}_{\b_1}(x))=\Lambda^x_{\beta_1,\b_2,\b_3}.
\end{equation} Equation \eqref{fibre identification} shows the connection between the set of $\beta_1$-expansions of a given $x$ and its vertical fibre. This connection is what allows us to prove Theorem \ref{Affine theorem}.

\begin{proposition}
\label{technical prop}
Let $\beta_1\in (1,\frac{1+\sqrt{5}}{2}).$ Then there exists $\delta=\delta(\beta_1)>0$ such that for any $\beta_2,\beta_3\in (1,1+\delta)$ and $x\in\widetilde{\mathcal{O}}_{\beta_1}$ we have $$\pi_{\beta_2,\beta_3}(B(\omega^{1},(\omega^{-1}_i)_{i=1}^{\infty}))<0 \textrm{ and } 0<\pi_{\beta_2,\beta_3}(B(\omega^{-1},(\omega^{1}_i)_{i=1}^{\infty})).$$
\end{proposition}

\begin{proof}
Let us start by fixing $\beta_1\in(1,\frac{1+\sqrt{5}}{2})$ and let $n(\beta_1)$ be as in the statement of Proposition \ref{important prop}. Let $\delta'>0$ be sufficiently small such that if $\beta_2,\beta_3\in(1,1+\delta'),$ then
\begin{equation}
\label{beathalf}
\sum_{i=1}^{|(\epsilon_i)|}\frac{\epsilon_i}{\beta_2^{|(\epsilon_1)_{j=1}^i|_{-1}}\beta_3^{|(\epsilon_1)_{j=1}^i|_{1}}}\geq \frac{1}{2}
\end{equation}whenever $(\epsilon_i)\in\{-1,1\}^*$ satisfies $|(\epsilon_i)|\leq n(\beta)$ and $|(\epsilon_i)|_{1}>|(\epsilon_i)|_{-1}.$ Such a $\delta'$ exists since for any $(\epsilon_i)$ satisfying these properties we have $$\sum_{i=1}^{|(\epsilon_i)|}\frac{\epsilon_i}{1^{|(\epsilon^1)_{j=1}^i|_{-1}}1^{|(\epsilon^1)_{j=1}^i|_{1}}}=\sum_{i=1}^{|(\epsilon_i)|}\epsilon_i=|(\epsilon_i)|_1-|(\epsilon_i)|_{-1}\geq 1>\frac{1}{2},$$ and strict inequality is preserved in a neighbourhood of $1$. For the same value of $\delta'$ we have
\begin{equation}
\label{losehalf}
\sum_{i=1}^{|(\epsilon_i)|}\frac{\epsilon_i}{\beta_2^{|(\epsilon_1)_{j=1}^i|_{-1}}\beta_3^{|(\epsilon_1)_{j=1}^i|_{1}}}\leq  -\frac{1}{2}
\end{equation}whenever $(\epsilon_i)\in\{-1,1\}^*$ satisfies $|(\epsilon_i)|\leq n(\beta)$ and $|(\epsilon_i)|_{-1}>|(\epsilon_i)|_{1}.$ Suppose $\beta_2,\beta_3\in(1,1+\delta'),$ then
\begin{align*}
\pi_{\beta_2,\beta_3}(B(\omega^{-1},(\omega_{i}^{1})_{i=1}^{\infty}))&= \pi_{\beta_2,\beta_3}(B(\omega^{-1})) + \sum_{i=0}^{\infty}\frac{\pi_{\beta_2,\beta_3}(B(\omega^1_{i+1}))}{\beta_2^{|\omega^{-1}|_{-1}+\sum_{j=0}^i|\omega^1_j|_{-1}}\beta_3^{|\omega^{-1}|_{1}+\sum_{j=0}^i|\omega^1_j|_{1}}}\\
& \geq -n(\beta_1) + \sum_{i=0}^{\infty}\frac{1}{2\beta_2^{|\omega^{-1}|_{-1}+\sum_{j=0}^i|\omega^1_j|_{-1}}\beta_3^{|\omega^{-1}|_{1}+\sum_{j=0}^i|\omega^1_j|_{1}}}\\
& \geq -n(\beta_1) + \sum_{i=0}^{\infty}\frac{1}{2\max(\beta_2,\beta_3)^{|\omega^{-1}|+\sum_{j=0}^i|\omega^1_j|}}\\
& \geq -n(\beta_1) + \sum_{i=0}^{\infty}\frac{1}{2\max(\beta_2,\beta_3)^{(i+1)n(\beta_1)}}\\
&\geq -n(\beta_1)+\frac{1}{2(\max(\beta_2,\beta_3)^{n(\beta_1)}-1)}.
\end{align*} In the first inequality we used \eqref{beathalf}. In the third inequality we used the fact that $|\omega^{-1}|\leq n(\beta_1),$ and $|\omega^{1}_i|\leq n(\beta_1)$ for all $i$. Summarising the above we have
\begin{equation}
\label{quack1}
-n(\beta_1)+\frac{1}{2(\max(\beta_2,\beta_3)^{n(\beta_1)}-1)} \leq \pi_{\beta_2,\beta_3}(B(\omega^{-1},(\omega_{i}^{1})_{i=1}^{\infty}))
\end{equation}
whenever $\beta_2,\beta_3\in(1,1+\delta')$. Similarly, one can show that if $\beta_2,\beta_3\in(1,1+\delta')$ then
\begin{equation}
\label{quack2}
\pi_{\beta_2,\beta_3}(B(\omega^{1},(\omega^{-1}_i)_{i=1}^{\infty}))\leq n(\beta_1)-\frac{1}{2(\max(\beta_2,\beta_3)^{n(\beta_1)}-1)}.
\end{equation}There exists $\delta''>0$ such that for $\beta_2,\beta_3\in(1,1+\delta'')$ we have
\begin{equation}
\label{quack3}
n(\beta_1)-\frac{1}{2(\max(\beta_2,\beta_3)^{n(\beta_1)}-1)}<0\textrm{ and } 0<n(\beta_1)+\frac{1}{2(\max(\beta_2,\beta_3)^{n(\beta_1)}-1)}.
\end{equation} Taking $\delta=\min(\delta',\delta''),$ we see that \eqref{quack1}, \eqref{quack2}, and \eqref{quack3} imply that for $\beta_2,\beta_3\in(1+\delta)$ we have $$\pi_{\beta_2,\beta_3}(B(\omega^{1},(\omega^{-1}_i)_{i=1}^{\infty}))<0 \textrm{ and } 0<\pi_{\beta_2,\beta_3}(B(\omega^{-1},(\omega^{1}_i)_{i=1}^{\infty})).$$ This completes our proof.
\end{proof}
In the proof of Proposition \ref{technical prop} the parameter $1/2$ appearing in \eqref{beathalf} and \eqref{losehalf} is an arbitrary choice. We could have replaced $1/2$ with any $c\in(0,1)$. It is not clear what an optimal choice of $c$ would be. What is more, the quantity $n(\beta_1)$ appearing in \eqref{quack1} and \eqref{quack2} is not necessarily optimal. In Section \ref{Section Explicit calculations} we see that for explicit an choice of $\beta_1$ these parameters can be improved upon to give a larger value of $\delta$.

The following corollary follows immediately from Proposition \ref{technical prop}.
\begin{corollary}
\label{technical corollary}
Let $\beta_1\in (1,\frac{1+\sqrt{5}}{2}).$ Then there exists $\delta=\delta(\b_1)>0$ such that for any $\beta_2,\beta_3\in (1,1+\delta)$ and $x\in\widetilde{\mathcal{O}}_{\beta_1}$ we have $$\pi_{\beta_2,\beta_3}(B(\lambda,\omega^{1},(\omega^{-1}_i)_{i=1}^{\infty}))<\pi_{\beta_2,\beta_3}(B(\lambda,\omega^{-1},(\omega^{1}_i)_{i=1}^{\infty}))$$ for all $\lambda\in\{T_0,T_1\}^{*}$.
\end{corollary}
\begin{proof}
Let $\b_1\in(1,\frac{1+\sqrt{5}}{2}).$ By Proposition \ref{technical prop} we know that for any $\beta_2,\beta_3\in (1,1+\delta)$ and $x\in\widetilde{\mathcal{O}}_{\beta_1}$we have $$\pi_{\beta_2,\beta_3}(B(\omega^{1},(\omega^{-1}_i)_{i=1}^{\infty}))<\pi_{\beta_2,\beta_3}(B(\omega^{-1},(\omega^{1}_i)_{i=1}^{\infty})).$$ It can be shown that the quantities $\pi_{\beta_2,\beta_3}(B(\lambda,\omega^{1},(\omega^{-1}_i)_{i=1}^{\infty}))$ and $\pi_{\beta_2,\beta_3}(B(\lambda,\omega^{-1},(\omega^{1}_i)_{i=1}^{\infty}))$ are the images of $\pi_{\beta_2,\beta_3}(B(\omega^{1},(\omega^{-1}_i)_{i=1}^{\infty}))$ and $\pi_{\beta_2,\beta_3}(B(\omega^{-1},(\omega^{1}_i)_{i=1}^{\infty}))$ under an orientation preserving affine map. Consequently the strict inequality is preserved.
\end{proof}We are now in a position to prove Theorem \ref{Affine theorem}.

\begin{proof}[Proof of Theorem \ref{Affine theorem}]
Let us fix $\beta_1\in(1,\frac{1+\sqrt{5}}{2})$ and let $\delta>0$ be as in the statement of Proposition \ref{technical prop}. Fix $\beta_2,\beta_3\in(1,1+\delta)$ and $x\in(\frac{-1}{\beta_1-1},\frac{1}{\beta_1-1}).$ By Lemma \ref{basic lemma} there exists $\lambda^0\in\{T_{-1},T_1\}^{*}$ such that $\lambda^0(x)\in\widetilde{\mathcal{O}}_{\beta_1}.$ Consider the interval $$[\pi_{\beta_2,\beta_3}(B(\lambda^0,(\omega^{-1}_i)_{i=1}^{\infty})),\pi_{\beta_2,\beta_3}(B(\lambda^0,(\omega^{1}_i)_{i=1}^{\infty}))].$$ Where $(\omega^{-1}_i)_{i=1}^{\infty}$ and $(\omega^{1}_i)_{i=1}^{\infty}$ are obtained by repeatedly applying Proposition \ref{important prop} to $\lambda^{0}(x)$ and its images. Recalling the proof of Proposition \ref{technical prop}, for this choice of $\delta$ we have $\pi_{\beta_2,\beta_3}(B(\omega^1_i))>1/2$ and $\pi_{\beta_2,\beta_3}(B(\omega^{-1}_i))<-1/2,$ as such the above interval is well defined and nontrivial. We will now show that this interval is contained within the fibre $\Lambda_{\beta_1,\beta_2,\beta_3}^x$. Let us fix $$y\in [\pi_{\beta_2,\beta_3}(B(\lambda^0,(\omega^{-1}_i)_{i=1}^{\infty})),\pi_{\beta_2,\beta_3}(B(\lambda^0,(\omega^{1}_i)_{i=1}^{\infty}))].$$ There are two cases to consider, either $$y\in [\pi_{\beta_2,\beta_3}(B(\lambda^0,(\omega^{-1}_i)_{i=1}^{\infty})),(B(\lambda^0,\omega^{-1}_1,(\omega^{1}_i)_{i=1}^{\infty}))]$$ or $$y\in[\pi_{\beta_2,\beta_3}(B(\lambda^0,\omega^{-1}_1,(\omega^{1}_i)_{i=1}^{\infty})),\pi_{\beta_2,\beta_3}(B(\lambda^0,(\omega^{1}_i)_{i=1}^{\infty}))].$$ The first interval is well defined and nontrivial by the same reasoning as that given above. The second interval is not necessarily well defined. However when it is not well defined, i.e., $\pi_{\beta_2,\beta_3}(B(\lambda^0,\omega^{-1}_1,(\omega^{1}_i)_{i=1}^{\infty}))>\pi_{\beta_2,\beta_3}(B(\lambda^0,(\omega^{1}_i)_{i=1}^{\infty}))$, then  $y$ is contained in the first interval. As such we can overlook this technicality.  In the first case we let $\lambda^{1}=(\lambda^0,\omega^{-1}),$ in the second case we let $\lambda^1=(\lambda^0,\omega^1).$ For the first case it is immediate that $$y\in[\pi_{\beta_2,\beta_3}(\lambda^1,(\omega^{-1}_i)_{i=1}^{\infty}),\pi_{\beta_2,\beta_3}(\lambda^1,(\omega^{1}_i)_{i=1}^{\infty})].$$ By Corollary \ref{technical corollary} we know that $$\pi_{\beta_2,\beta_3}(B(\lambda^0,\omega^{1},(\omega^{-1}_i)_{i=1}^{\infty}))<\pi_{\beta_2,\beta_3}(B(\lambda^0,\omega^{-1},(\omega^{1}_i)_{i=1}^{\infty})).$$ Therefore for the second case we also have $$y\in[\pi_{\beta_2,\beta_3}(\lambda^1,(\omega^{-1}_i)_{i=1}^{\infty}),\pi_{\beta_2,\beta_3}(\lambda^1,(\omega^{1}_i)_{i=1}^{\infty})].$$Now suppose we have constructed a sequence $\lambda^k\in\{T_{-1},T_1\}^{*}$ such that
\begin{equation}
\label{stepk}
y\in[\pi_{\beta_2,\beta_3}(B(\lambda^k,(\omega^{-1}_i)_{i=1}^{\infty})),\pi_{\beta_2,\beta_3}(B(\lambda^k,(\omega^{1}_i)_{i=1}^{\infty}))].
\end{equation}We now show how to construct $\lambda^{k+1}$ satisfying \eqref{stepk}. Again there are two cases to consider, either $$y\in [\pi_{\beta_2,\beta_3}(B(\lambda^k,(\omega^{-1}_i)_{i=1}^{\infty})),(B(\lambda^k,\omega^{-1}_1,(\omega^{1}_i)_{i=1}^{\infty}))]$$ or $$y\in[\pi_{\beta_2,\beta_3}(B(\lambda^k,\omega^{-1}_1,(\omega^{1}_i)_{i=1}^{\infty})),\pi_{\beta_2,\beta_3}(B(\lambda^k,(\omega^{1}_i)_{i=1}^{\infty}))].$$ The first interval is still well defined and nontrivial. The second interval is not necessarily well defined but this technicality can be overlooked for the same reason as that given before. In the first case we take $\lambda^{k+1}=(\lambda^k,\omega^{-1}),$ then we automatically have $$[\pi_{\beta_2,\beta_3}(B(\lambda^{k+1},(\omega^{-1}_i)_{i=1}^{\infty})),\pi_{\beta_2,\beta_3}(B(\lambda^{k+1},(\omega^{1}_i)_{i=1}^{\infty}))].$$ In the second case we take $\lambda^{k+1}=(\lambda^k,\omega^{1}).$ Applying Corollary \ref{technical corollary} as above we then have $$[\pi_{\beta_2,\beta_3}(B(\lambda^{k+1},(\omega^{-1}_i)_{i=1}^{\infty})),\pi_{\beta_2,\beta_3}(B(\lambda^{k+1},(\omega^{1}_i)_{i=1}^{\infty}))].$$ Thus we have completed our inductive step.

Continuing in this manner yields an infinite sequence $\lambda\in \Omega_{\beta_1}(x)$. Since the diameter of the interval appearing in \eqref{stepk} tends to zero as $k\to \infty$, it follows that $$y=\pi_{\beta_2,\beta_3}(B(\lambda)).$$ Since $y$ was arbitrary it follows that $$ [\pi_{\beta_2,\beta_3}(B(\lambda^0,(\omega^{-1}_i)_{i=1}^{\infty})),\pi_{\beta_2,\beta_3}(B(\lambda^0,(\omega^{1}_i)_{i=1}^{\infty}))]\subseteq \pi_{\b_2,\b_3}(B(\widetilde{\Omega}_{\b_1}(x))).$$ By \eqref{fibre identification} and Lemma \ref{Bijection lemma} it follows that $$ [\pi_{\beta_2,\beta_3}(B(\lambda^0,(\omega^{-1}_i)_{i=1}^{\infty})),\pi_{\beta_2,\beta_3}(B(\lambda^0,(\omega^{1}_i)_{i=1}^{\infty}))]\subseteq \Lambda^x_{\b_1,\b_2,\b_3}$$ as required.

To see that $(0,0)\in\Lambda_{\beta_1,\beta_2,\beta_3}^{0},$ we remark that if $x\in\widetilde{\mathcal{O}}_{\beta_1}$ then we do not require the initial map $\lambda^0$ which maps $x$ into $\widetilde{\mathcal{O}}_{\beta_1}.$ Consequently, for every $x\in \widetilde{\mathcal{O}}_{\beta_1}$ the fibre $\Lambda_{\beta_1,\beta_2,\beta_3}^x$ contains the interval $$[\pi_{\beta_2,\beta_3}(B((\omega^{-1}_i)_{i=1}^{\infty})),\pi_{\beta_2,\beta_3}(B((\omega^{1}_i)_{i=1}^{\infty}))].$$ By Proposition \ref{technical prop} this interval contains a neighbourhood of zero. Since $0$ is contained in the interior of $\widetilde{\mathcal{O}}_{\beta_1}$ it follows that $(0,0)\in\Lambda_{\beta_1,\beta_2,\beta_3}^{0}$.

\end{proof}

\section{An explicit calculation}
\label{Section Explicit calculations}
In this section we fix $\beta^*\approx 1.4656$ the appropriate root of $x^3-x^2-1=0$. A simple calculation yields $$\mathcal{O}_{\beta^*}=[0.872\ldots,1.276\ldots].$$ In Table \ref{tab:table-name} we include a list of intervals that partition $\mathcal{O}_{\beta^*}$ along with the corresponding sequences $\omega^0$ and $\omega^1$ which satisfy the conclusions of Proposition \ref{important prop} for those elements within each interval.

\begin{table}
 \begin{tabular}{|c c|| c c|}
 \hline
 Interval & $\omega^0$ & Interval & $\omega^1$ \\ [0.5ex]
 \hline\hline
 [0.872\ldots,0.959\ldots] & $(T_1,T_0,T_0,T_0)$ & [1.188\ldots,1.276\ldots] & $(T_0,T_1,T_1,T_1)$\\
 \hline
 [0.959\ldots,1.087\ldots] & $(T_1,T_0,T_0)$ & [1.061\ldots,1.188\ldots]  & $(T_0,T_1,T_1)$\\
 \hline
 [1.087\ldots,1.128\ldots] & $(T_1,T_0,T_1,T_0,T_0,T_0)$ & [1.020\ldots,1.061\ldots]  & $(T_0,T_1,T_0,T_1,T_1,T_1)$\\
 \hline
 [1.128\ldots,1.188\ldots] & $(T_1,T_0,T_1,T_0,T_0)$ & [0.960\ldots,1.020\ldots] & $(T_0,T_1,T_0,T_1,T_1)$ \\
\hline
 [1.188\ldots, 1.208\ldots] & $(T_1,T_1,(T_0)^6)$ & [0.940\ldots,0.960\ldots] & $(T_0,T_0,(T_1)^6)$\\
\hline
 [1.208\ldots,1.236\ldots] & $(T_1,T_1,(T_0)^5)$ & [0.912\ldots,0.940\ldots] & $(T_0,T_0,(T_1)^5)$\\
\hline
 [1.236\ldots,1.276\ldots] & $(T_1,T_1,(T_0)^4)$  & [0.872\ldots,0.912\ldots] & $(T_0,T_0,(T_1)^4)$\\ [1ex]
 \hline
\end{tabular}
\vspace{5mm}
\caption{A partition of the interval $\mathcal{O}_{\beta^{*}}$ and the corresponding $\omega^{0}$ and $\omega^1$.}
\label{tab:table-name}
\end{table}
Upon examination of Table \ref{tab:table-name} we observe that $|\omega^0|\leq 8$ and $|\omega^1|\leq 8$ for all $\omega^0$ and $\omega^1$. As such we can take $n(\beta^*)=8$. It follows from the proof of Proposition \ref{freq prop}, that for any $x\in(0,\frac{1}{\beta^{*}-1})$ and $p\in[7/16,9/16],$ there exists an expansion of $x$ in base $\beta^{*}$ such that the digit zero occurs with frequency $p$.

We now consider Theorem \ref{Affine theorem} and show how one can explicitly calculate the parameter $\delta$ appearing in its statement. Note that $$\widetilde{\mathcal{O}}_{\beta^*}=[-0.403\ldots,0.403\ldots].$$  We start by pointing out Table \ref{tab:table-name2}. This table lists a collection of intervals that partition $\widetilde{\mathcal{O}}_{\beta^*}$ along with the corresponding sequences $\omega^{-1}$ and $\omega^1$ which satisfy the conclusions of Proposition \ref{important prop} for those elements within each interval. We remark that Table \ref{tab:table-name2} can be obtained from Table \ref{tab:table-name} by a simple change of coordinates.
\begin{table}
 \begin{tabular}{|c c|| c c|}
 \hline
 Interval & $\omega^{-1}$ & Interval & $\omega^1$ \\ [0.5ex]
 \hline\hline
 [-0.403\ldots,-0.229\ldots] & $(T_1,T_{-1},T_{-1},T_{-1})$ & [0.229\ldots,0.403\ldots] & $(T_{-1},T_1,T_1,T_1)$\\
 \hline
 [-0.229\ldots,0.026\ldots] & $(T_1,T_{-1},T_{-1})$ & [-0.026\ldots,0.229\ldots] & $(T_{-1},T_1,T_1)$\\
 \hline
 [0.026\ldots,0.108\ldots] & $(T_1,T_{-1},T_1,T_{-1},T_{-1},T_{-1})$ & [-0.108\ldots, -0.026\ldots]  & $(T_{-1},T_1,T_{-1},T_1,T_1,T_1)$\\
 \hline
 [0.108\ldots,0.228\ldots] & $(T_1,T_{-1},T_1,T_{-1},T_{-1})$ & [-0.228\ldots, -0.108\ldots]& $(T_{-1},T_1,T_{-1},T_1,T_1)$ \\
\hline
 [0.228\ldots, 0.268\ldots] & $(T_1,T_1,(T_{-1})^6)$ &  [-0.268\ldots, -0.228\ldots] & $(T_{-1},T_{-1},(T_1)^6)$\\
\hline
 [0.268\ldots,0.324\ldots] & $(T_1,T_1,(T_{-1})^5)$ & [-0.324\ldots,- 0.268\ldots] & $(T_{-1},T_{-1},(T_1)^5)$\\
\hline
 [0.324\ldots,0.403\ldots] & $(T_1,T_1,(T_{-1})^4)$  & [-0.403\ldots,-0.324\ldots] & $(T_{-1},T_{-1},(T_1)^4)$\\ [1ex]
 \hline
\end{tabular}
\vspace{5mm}
\caption{A partition of the interval $\widetilde{\mathcal{O}}_{\beta^{*}}$ and the corresponding $\omega^{-1}$ and $\omega^1$.}
\label{tab:table-name2}
\end{table}

The crucial step in the proof of Theorem \ref{Affine theorem} is Proposition \ref{technical prop}. The $\delta$ appearing in this statement is the same $\delta$ appearing in the statement of Theorem \ref{Affine theorem}. As such to determine a $\delta$ so that the conclusions of Theorem \ref{Affine theorem} are satisfied, we need to calculate a $\delta$ such that if $\beta_2,\beta_3\in (1,1+\delta)$ and $x\in\widetilde{\mathcal{O}}_{\b^*}$ then
\begin{equation}
\label{need to show}
\pi_{\beta_2,\beta_3}(B(\omega^{1},(\omega^{-1}_i)_{i=1}^{\infty}))<0 \textrm{ and } 0<\pi_{\beta_2,\beta_3}(B(\omega^{-1},(\omega^{1}_i)_{i=1}^{\infty})).
\end{equation} Let $$A_{-1}=\{\omega^{-1}:\omega^{-1} \textrm{ appears in Table \ref{tab:table-name2}}\}\textrm{ and }A_1=\{\omega^1:\omega^1 \textrm{ appears in Table \ref{tab:table-name2}}\}.$$ We will explicitly construct a $\delta$ such that if $\beta_2,\beta_3\in (1,1+\delta)$ then
\begin{equation}
\label{will show}
\pi_{\beta_2,\beta_3}(B(a,(b_i)_{i=1}^{\infty}))<0 \textrm{ and } 0<\pi_{\beta_2,\beta_3}(B(c,(d_i)_{i=1}^{\infty})),
\end{equation} for any $a\in A_1$ and $(b_i)\in A_{-1}^{\mathbb{N}},$ and for any $c\in A_{-1}$ and $(d_i)\in A_{1}^{\mathbb{N}}.$ Clearly \eqref{will show} implies \eqref{need to show}.

The following lemma makes determining a $\delta$ for which \eqref{will show} holds far more tractable.

\begin{lemma}
\label{geometric progression}
Let $\mathcal{D}=\{\kappa_l\}\subseteq\{-1,1\}^{*}$ be a finite set consisting of strings of digits (possibly of different lengths). Then $$\min_l \pi_{\beta_2,\beta_3}((\kappa_l)^{\infty})\leq \pi_{\beta_2,\beta_3}((a_i)_{i=1}^{\infty})\leq \max_l \pi_{\beta_2,\beta_3}((\kappa_l)^{\infty})$$ for any $(a_i)\in\mathcal{D}^{\mathbb{N}}.$
\end{lemma}

\begin{proof}
Let $$J=\Big[\min_l \pi_{\beta_2,\beta_3}((\kappa_l)^{\infty}),\max_l \pi_{\beta_2,\beta_3}((\kappa_l)^{\infty})\Big].$$  Fix a sequence $(b_i)\in \mathcal{D}^{\mathbb{N}}$ such that $\pi_{\beta_2,\beta_3}((b_i))\in J$ (one could simply take the sequence $(b_i)=(\kappa_l)^{\infty}$ for any $l$), and let $(a_i)\in \mathcal{D}^{\mathbb{N}}$ be arbitrary. Consider the point $ \pi_{\beta_2,\beta_3}((a_1,(b_i))).$ Then $\pi_{\beta_2,\beta_3}((a_1,(b_i)))\in J$. This is because both $\pi_{\beta_2,\beta_3}((b_i))$ and $\pi_{\beta_2,\beta_3}((a_1)^{\infty})$ are contained in $J$ and
\begin{equation}
\label{contraction}
|\pi_{\beta_2,\beta_3}(a_1,(b_i))- \pi_{\beta_2,\beta_3}((a_1)^{\infty})|\leq |\pi_{\beta_2,\beta_3}((b_i))- \pi_{\beta_2,\beta_3}((a_1)^{\infty})|.
\end{equation}  Equation \eqref{contraction} holds because prefixing $(b_i)$ by $a_1$ corresponds to applying a uniformly contracting similarity to $\pi_{\beta_2,\beta_3}((b_i)),$ where this similarity has its unique fixed point at $\pi_{\beta_2,\beta_3}((a_1)^{\infty})$.

Repeating the above argument it follows that for any $n\in \mathbb{N}$ we have $\pi_{\beta_2,\beta_3}((a_i)_{i=1}^n,(b_i))\in J.$ Since $$\pi_{\beta_2,\beta_3}((a_i)_{i=1}^n,(b_i))\to \pi_{\beta_2,\beta_3}((a_i)_{i=1}^{\infty})$$ and $J$ is closed, we have $\pi_{\beta_2,\beta_3}((a_i)_{i=1}^{\infty})\in J$ as required.
\end{proof}

It is a consequence of Lemma \ref{geometric progression} that to calculate a $\delta$ such that \eqref{will show} holds, it suffices to determine a $\delta$ such that for all $\beta_2,\beta_3\in(1,1+\delta)$ we have
\begin{equation}
\label{will show2}
\max_{a\in A_1, b\in A_{-1}} \pi_{\beta_2,\beta_3}(B(a,(b)^{\infty}))<0 \textrm{ and } 0<\min_{c\in A_{-1}, d\in A_{1}} \pi_{\beta_2,\beta_3}(B(c,(d)^{\infty})).
\end{equation}
Since there are only finitely many elements in $A_{-1}$ and $A_1,$ to determine a $\delta$ for which \eqref{will show2} holds one only has to consider finitely many inequalities. Inputting each of these inequalities into a computer yields $\delta=0.041.$ Consequently if $\beta_2,\beta_3\in(1,1.041)$ then \eqref{will show2} holds and by Proposition \ref{technical prop} and Theorem \ref{Affine theorem} the fibre $\Lambda_{\beta^*,\beta_2,\beta_3}^x$ contains an interval for all $x\in (\frac{-1}{\b^*-1},\frac{1}{\beta-1}).$ In Figure \ref{figc} we include a plot of $\Lambda_{\b^*,1.03,1.04}.$

\begin{figure}[h]
\includegraphics[width=10cm, height=7cm]{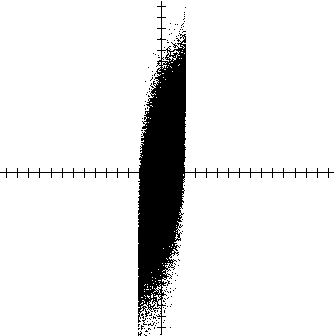}
\centering
\caption{A plot of  $\Lambda_{\b^*,1.03,1.04}.$ For each $x\in(\frac{-1}{\b^*-1},\frac{1}{\beta^*-1})$ the fibre $\Lambda_{\b^*,1.03,1.04}^x$ contains an interval.}
\label{figc}
\end{figure}

\section{Remarks}
\label{remarks}
We finish this paper by making some remarks and posing questions.
\begin{remark}
Note that when $\beta=\frac{1+\sqrt{5}}{2}$ it can be shown that $$\Sigma_{\frac{1+\sqrt{5}}{2}}(1)=\Big\{(10)^{\infty}, ((10)^k0(1)^{\infty}), (10)^k11(0)^{\infty}: k\geq 0\Big\}.$$ Moreover, for $\beta\in (\frac{1+\sqrt{5}}{2},2)$ it can be shown that $$\Sigma_{\beta}\Big(\frac{\beta}{\beta^2-1}\Big)=\{(10)^{\infty}\}.$$ Consequently, we see that statements $2$ and $3$ from Theorem \ref{frequency theorem} cannot be extended past the parameter $\frac{1+\sqrt{5}}{2}.$ Thus these statements are optimal.

Similarly, for the digit set $\{-1,1\}$ one can construct nontrivial $x$ such that $\widetilde{\Omega}_{\frac{1+\sqrt{5}}{2}}(x)$ is infinite countable, and for $\b_1\in(\frac{1+\sqrt{5}}{2},2)$ examples of nontrivial $x$ for which $\widetilde{\Omega}_{\frac{1+\sqrt{5}}{2}}(x)$ is a singleton set. For these particular choices of $x$ it is clear that the vertical fibre $\Lambda^x_{\beta_1,\beta_2,\beta_3}$cannot contain an interval. Consequently one cannot improve upon the interval $(1,\frac{1+\sqrt{5}}{2})$ appearing in the statement of Theorem \ref{Affine theorem}.

\end{remark}

\begin{remark}
Statement $1$ from Theorem \ref{frequency theorem} tells us that for $\beta\in(1,\beta_{KL})$ every $x\in(0,\frac{1}{\beta-1})$ has a simply normal expansion. It is natural to ask whether one can improve upon $\beta_{KL}$. In \cite{JSS} Jordan, Shmerkin, and Solomyak proved the following  result.
\begin{proposition}
\label{JSS lemma}
Let $\beta_{T}\approx 1.80193$ be the unique solution to the equation $1=\frac{1}{\beta}+\sum_{i=1}^{\infty}\frac{1}{\beta^{2i}}.$ Then for any $\beta>\beta_{T}$ there exists $x\in \mathcal{U}_{\beta}$ such that its unique $\beta$-expansion is not simply normal. Moreover for any $\beta\in(1,\beta_{T}]$ we have that every $(\epsilon_i)\in\widetilde{\mathcal{U}}_{\beta}\setminus\{(0)^{\infty},(1)^{\infty}\}$ is simply normal.
\end{proposition}
This leaves a closed interval of size $\beta_T-\beta_{KL}\approx 0.01473$ for which we don't know whether every $x\in(0,\frac{1}{\beta-1})$ has a simply normal $\beta$-expansion. It would be interesting to determine the behaviour within this interval. The author conjectures that for $\beta\in[\beta_{KL},\beta_{T}]$ every $x\in(0,\frac{1}{\beta-1})$ has a simply normal $\beta$-expansion.
\end{remark}

\begin{remark}
In this paper we have only considered simply normal expansions. It is natural to wonder about normal expansions. Recall that a sequence $(\epsilon_i)\in\{0,1\}^{\mathbb{N}}$ is normal if for every finite block $(\delta_1,\ldots,\delta_k)\in\{0,1\}^{*}$ we have $$\lim_{n\to\infty}\frac{\#\{1\leq i \leq n: \epsilon_i=\delta_1,\ldots, \epsilon_{i+k-1}=\delta_k\}}{n}=\frac{1}{2^{k}}.$$ A natural question to ask is whether there exists $c>0$ such that if $\beta\in(1,1+c)$ then every $x\in\Big(0,\frac{1}{\beta-1}\Big)$  has a normal expansion. This question was originally posed to the author by Kempton \cite{Kem}. The author suspects that such a $c$ does not exist but we cannot prove this. A natural obstruction to proving the nonexistence of such a $c$ is that there exists $c'>0$ such that for any $\beta\in(1,1+c'),$ every $x\in(0,\frac{1}{\beta-1})$ has a $\beta$-expansion that contains all finite blocks of digits. The existence of such a $c'$ was originally proved by Erd\H{o}s and Komornik \cite{ErdKom}. Consequently, to prove the nonexistence of such a $c$ one would have to prove that for every $\beta$ sufficiently close to one, there exists an $x\in(0,\frac{1}{\beta-1})$ such that for every $(\epsilon_i)\in\Sigma_{\beta}(x)$ there exists a block of digits which do not occur with the desired frequency. This seems like a difficult problem.
\end{remark}

\begin{remark}
In Theorem \ref{exceptional theorem} we proved that $$\lim_{\beta\nearrow 2}\dim_{H}\Big(\Big\{x: x \textrm{ has no simply normal } \beta\textrm{-expansion}\Big\}\Big)=1.$$ It is natural to ask whether $$\dim_{H}\Big(\Big\{x: x \textrm{ has no simply normal } \beta\textrm{-expansion}\Big\}\Big)<1$$ for all $\beta\in(1,2)$. A solution to this question would likely involve the study of those $x$ for which $\Sigma_{\beta}(x)$ is uncountable yet every element of $\Sigma_{\beta}(x)$ fails to be simply normal. Studying this set seems like a difficult task. Indeed for $\beta$ close to $2$ it is unclear whether this set is nonempty.

\end{remark}
\begin{remark}
In the proof of Theorem \ref{Affine theorem} we explicitly constructed an interval appearing in the fibre $\Lambda_{\b_1,\b_2,\b_3}^x$. For each $y$ in this interval we construct a $\lambda\in\widetilde{\Omega}_{\b_1}(x)$ such that $\pi_{\b_2,\b_3}(B(\lambda))=y$. The method by which we construct this $\lambda$ bears a strong resemblance to the way one normally constructs $\beta$-expansions. The author wonders whether a refinement of the argument given in the proof of Proposition \ref{important prop} would yield an algorithm by which we can construct many $\lambda\in\widetilde{\Omega}_{\b_1}(x)$ such that $\pi_{\b_2,\b_3}(B(\lambda))=y,$ and for which we have a lot of control over the frequency of the $T_{-1}$'s and $T_1$'s that appear in $\lambda$. With such an algorithm the author expects one could adapt the proof of Theorem \ref{Affine theorem} to give new examples of self-affine sets in three dimensions with nonempty interior. Possibly this method could be extended to $n$-dimensional self-affine sets.
\end{remark}

\noindent \textbf{Acknowledgements.} The author would like to thank Tom Kempton and Ben Pooley for some useful discussions. The author would also like to thank Nikita Sidorov for suggesting the technique of studying the fibres of self-affine sets by considering them as projections of the set of $\beta$-expansions. This research was supported by the EPSRC grant EP/M001903/1.

\end{document}